\RequirePackage{fix-cm}
\documentclass[smallextended]{svjour3}       
\smartqed  
\usepackage{graphicx}
\usepackage{amssymb,amsmath,amsbsy,color}

\newcommand{\ul}{\underline}
\newcommand{\ee}{\textnormal{\bf e}}
\newcommand{\ii}{\textnormal{\bf i}}
\newcommand{\dd}{\textnormal{d}}
\newcommand{\diag}{\textnormal{diag}}
\newcommand{\supp}{\textnormal{supp}}
\newcommand{\underbr}[2]{\underbrace{#1}_{\displaystyle #2}}

\newtheorem{assumption}{Assumption}

\newif\ifproofinsert
\proofinserttrue
\newsavebox\myproofbox
\ifproofinsert
\else
\renewenvironment{proof}{%
\begin{lrbox}{\myproofbox}\begin{minipage}{\linewidth}%
}{\end{minipage}\end{lrbox}}
\fi

\newcommand{\hn}{h} 
\newcommand{\nh}{n_h} 
\newcommand{\Spkh}{S_{p,k,h}}
\newcommand{\Spph}{S_{p,p,h}}
\newcommand{\Sph}{S_{p,h}}
\newcommand{\Sphper}{\widehat{S}_{p,h}}
\newcommand{\Sphhper}{\widehat{S}_{p,\frac{h}{2}}}
\newcommand{\upkh}{u_{p,k,h}}
\newcommand{\uph}{u_{p,h}}
\newcommand{\uphh}{u_{p,\frac{h}{2}}}
\newcommand{\uphhvec}{\ul{u}_{p,\frac{h}{2}}}
\newcommand{\uphvec}{\ul{u}_{p,h}}

\newcommand{\vph}{v_{p,h}}

\newcommand{\tiSph}{\widetilde{S}_{p,h}}
\newcommand{\Shone}{{S}_{1,h}}
\newcommand{\Shtwo}{{S}_{2,h}}
\newcommand{\Shthree}{{S}_{3,h}}
\newcommand{\Shfour}{{S}_{4,h}}
\newcommand{\tiShone}{\widetilde{S}_{1,h}}
\newcommand{\tiShtwo}{\widetilde{S}_{2,h}}
\newcommand{\tiShthree}{\widetilde{S}_{3,h}}
\newcommand{\tiShfour}{\widetilde{S}_{4,h}}

\newcommand{\vphvec}{\ul{v}_{p,h}}

\newcommand{\wph}{w_{p,h}}

\newcommand{\bspl}{\varphi_{p,h}}
\newcommand{\bsplper}{\widehat{\varphi}_{p,h}}
\newcommand{\bsplperpme}{\widehat{\varphi}_{p-1,h}}
\newcommand{\bsplti}{\widetilde{\varphi}_{p,h}}
\newcommand{\OmegaAB}{\Omega}
\newcommand{\OmegaPhys}{\hat\Omega}
\newcommand{\Ho}{H^1_{\circ}}

\usepackage{color}
\definecolor{darkgreen}{rgb}{0,0.5,0}

\begin{document}

\title{Approximation error estimates and inverse inequalities for B-splines of maximum smoothness}
\titlerunning{Approximation error and inverse inequalities for splines of maximum smoothness}
\author{Stefan Takacs, Thomas Takacs}

\institute{Stefan Takacs \at
						Department of Computational Mathematics,
           Johannes Kepler University Linz, Austria\\
           \email{stefan.takacs@numa.uni-linz.ac.at}\\
           The work of this author was funded by the Austrian Science Fund (FWF): J3362-N25.
       \and
       Thomas Takacs \at
		Department of Mathematics, 
           University of Pavia, Italy\\
           \email{thomas.takacs@unipv.it}\\
	The work of this author was partially funded by the European Research Council through 
	the FP7 Ideas Consolidator Grant HIgeoM.
}

\date{\today}

\journalname{Numerische Mathematik}

\maketitle

\begin{abstract}
  In this paper, we develop approximation error estimates as well as corresponding
  inverse inequalities for B-splines of maximum smoothness, where both the function to
  be approximated and the approximation error are measured in standard Sobolev norms
  and semi-norms.
	The presented approximation error estimates do not depend on the polynomial degree
	of the splines but only on the grid size.
	
	We will see that the approximation lives in a subspace of the classical B-spline
	space. We show that for this subspace, there is an inverse inequality which is also
	independent of the polynomial degree. As the approximation error estimate and
	the inverse inequality show complementary behavior, the results shown in this
	paper can be used to construct fast iterative methods for solving problems arising
	from isogeometric discretizations of partial differential equations.
\end{abstract}

\section{Introduction}\label{sec:intro}

The objective of this paper is to prove approximation error estimates 
as well as corresponding inverse estimates for B-splines of maximum smoothness. 
The presented approximation error estimates do not depend on the degree of the splines but only 
on the grid size. All bounds are given in terms of classical Sobolev norms and semi-norms. 

In approximation theory, B-splines have been studied for a long time and many properties 
are already well known. We do not go into the details of the existing results but present 
the results of importance for our study throughout this paper. 

The emergence of Isogeometric Analysis, introduced in~\cite{Hughes:2005}, sparked new interest 
in the theoretical properties of B-splines. Since isogeometric Galerkin methods are aimed at 
solving variational formulations of differential equations, approximation properties 
measured in Sobolev norms need to be studied. 

The results presented in this paper improve the results given in
\cite{Schumaker:1981,devore:1993,Bazilevs:2006} by explicitly studying the dependence
on the polynomial degree~$p$. Such an analysis was done in~\cite{daVeiga:2011}. However,
the results there do not cover (for~$p\geq 2$) the most important case of B-splines
of maximum smoothness~$k=p-1$. It turns out that the methods established in~\cite{daVeiga:2011} 
for proving those bounds are not suitable in that case. Therefore, we develop a framework based on 
Fourier analysis to prove rigorous bounds for $k=p-1$, which has the limitation that it is 
only applicable for uniform grids. 

Unlike the aforementioned papers we only consider approximation with B-splines 
in the parameter domain within the framework of Isogeometric Analysis. 
A generalization of the results to NURBS as well as the 
introduction of a geometry mapping, as 
presented in \cite{Bazilevs:2006}, is straightforward and does not lead to any additional 
insight. 

Note that a detailed study of direct and inverse estimates may lead to a deeper 
understanding of isogeometric multigrid methods and give insight to suitable 
preconditioning methods. We refer to \cite{Garoni:2014,Donatelli:2015}, where 
similar techniques were used. 

\subsection{The main results}

We now go through the main results of this paper.
For simplicity, we consider the case of one dimension first, where $\OmegaAB = (a,b)$
with $a<b$ is the open \emph{parameter domain}.
For this domain we can introduce a \emph{uniform grid} by subdividing
$\OmegaAB$ into \emph{elements} (subintervals) of length $\hn$. The setup of a
uniform grid is only possible if
\begin{equation*}
		\nh := \hn^{-1}(b-a)\in \mathbb{N},
\end{equation*}
where $\mathbb{N}:= \{1,2,3,\ldots\}$. 
In other words, the grid size $h$ has to be chosen such that $\nh$,
the number of subintervals, is an integer. We will assume this
implicitly throughout the paper.
On these grids we can introduce spaces of spline functions.
\begin{definition}
	The space of spline functions on the domain $\OmegaAB$ of degree $p\in \mathbb{N}_0:=\{0,1,2,\ldots\}$ and
	continuity $k\in\{-1,0,1,2,\ldots\}$ over the uniform grid 
	of size $\hn$ is given by 
\begin{equation*}
	S_{p,k,h}(\OmegaAB) := \left\{ u \in H^k(\OmegaAB): \; u |_{(a+hj,a+h(j+1)]} \in \mathbb{P}^p \mbox{ for all } j=0,\ldots,\nh-1 \right\},
\end{equation*}
where $\mathbb{P}^p$ is the space of polynomials of degree $p$.
\end{definition}
Here and in what follows, $L^2(\OmegaAB)$ and $H^r(\OmegaAB)$ denote the standard Lebesque and
Sobolev spaces with norms $\|\cdot\|_{L^2(\OmegaAB)}$, $\|\cdot\|_{H^r(\OmegaAB)}$ and
semi-norms $|\cdot|_{H^r(\OmegaAB)}$.
Moreover, let $(\cdot,\cdot)_{L^2(\OmegaAB)}$ be the standard scalar product for $L^2(\OmegaAB)$ and
\begin{equation*}
	(u,v)_{H^r(\OmegaAB)} := \left(\frac{\partial^r}{\partial x^r} u,\frac{\partial^r}{\partial x^r} v\right)_{L^2(\OmegaAB)}
\end{equation*}
be the scalar product for $H^r(\OmegaAB)$, where $\frac{\partial^r}{\partial x^r}$ denotes the $r$-th derivative. We then have $|u|^2_{H^r(\OmegaAB)} := (u,u)_{H^r(\OmegaAB)}$ as well as 
\begin{equation*}
	\|u\|^2_{H^r(\OmegaAB)} := \|u\|^2_{L^2(\OmegaAB)} + \sum^r_{s=1} |u|^2_{H^s(\OmegaAB)}
\end{equation*}
for all $r\in\mathbb{N}_0:=\{0,1,2,\ldots\}$.

Using standard trace theorems, we obtain that for $k>0$ the
space $S_{p,k,h}(\OmegaAB)$ is the space of all $k-1$ times continuously differentiable
functions ($C^{k-1}(\Omega)$-functions), which are polynomials of degree $p$ on each
element of the uniform grid on $\OmegaAB$. For $k=0$, there is no continuity condition, i.e.,
the space $S_{p,0,h}(\OmegaAB)$ is the space of piecewise polynomials of degree $p$. 

For $k>p$, the spline spaces reduce to spaces  
of global polynomials. So, the largest possible choice for $k$ without having this
effect is $k=p$. Therefore we call B-splines with $k=p$ B-splines of \emph{maximum smoothness}. 
As we are mostly interested in this case, here and in what follows,
we will use $\Sph(\OmegaAB):=\Spph(\OmegaAB)$. 

The main result of this paper is the following.

\newcommand{\citeThrmApprox}{1} 
\begin{theorem}\label{thrm:approx}
    For all $u\in H^1(\OmegaAB)$, all grid sizes $h$
    and each degree $p\in\mathbb{N}$ with ${h\,p < |\OmegaAB| = b-a}$,
    there is a spline approximation $\uph\in \Sph(\OmegaAB)$ such that
    \begin{equation}\label{eq:thrm:approx}
			    	\|u-\uph\|_{L^2(\OmegaAB)} \le \sqrt{2}\; \hn |u|_{H^1(\OmegaAB)}
    \end{equation}
    is satisfied.
\end{theorem}

Note that, in contrast to the existing results presented in the next subsection, this theorem achieves
two goals, it covers the case of maximum smoothness and gives a uniform estimate for all polynomial degrees~$p$.

\begin{remark}
		Obviously $\Spkh(\OmegaAB) \supseteq \Sph(\OmegaAB)$ for all $0\le k < p$. So,
		Theorem~\ref{thrm:approx} is also valid in that case. However, for this case there might
		be better estimates for these larger B-spline spaces.
		Moreover, Theorem~\ref{thrm:approx} is also satisfied in
		the case of having repeated knots, as this is just a local
		reduction of the continuity (which enlarges the corresponding space of spline
		functions).
\end{remark}

In Section~\ref{sec:reduced}, we will introduce a
subspace $\tiSph(\OmegaAB) \subseteq \Sph(\OmegaAB)$ (cf. Definition~\ref{defi:Ssymm}) and
show that the spline approximation
is even in that subspace (cf. Corollary~\ref{cor:approx:nonper}).
Moreover, we show also a corresponding \emph{inverse inequality} for $\tiSph(\OmegaAB)$ (cf. Theorem~\ref{thrm:inverse}
in Section~\ref{sec:inverse}),
i.e., we will show that
\begin{equation*}
	|\uph|_{H^1(\OmegaAB)} \le 2 \sqrt{3} \hn^{-1} \|\uph\|_{L^2(\OmegaAB)}
\end{equation*}
is satisfied for all grid sizes $h$, each $p\in \mathbb{N}$ and all $\uph\in \tiSph(\OmegaAB)$.

We will moreover show that both the approximation error estimate and the
inverse inequality are \emph{sharp up to constants}
(Corollaries~\ref{corr:sharp1} and~\ref{corr:sharp2}).

\begin{remark}\label{rem:counterexample}
	This inverse inequality does not extend to the whole space $\Sph(\OmegaAB)$.
	Here it is easy to find a counterexample: Let $\OmegaAB = (0,1)$. 
	The function $\uph$, given by
	\begin{equation*}
		\uph(x) = \left\{
			\begin{array}{ll}
				(1-x/\hn)^p  & \mbox{\qquad for $x \in [0,\hn)$}\\
				0 & \mbox{\qquad for $x\in [\hn,1]$},
			\end{array}
			\right.
	\end{equation*}
	is a member of the space $\Sph(0,1)$. Straight-forward computations yield
	\begin{equation*}
        \frac{|\uph|_{H^1(0,1)}}{\|\uph\|_{L^2(0,1)}} = \sqrt{\frac{2p+1}{2p-1}} \;p \;\hn^{-1},
	\end{equation*}
	which cannot be bounded from above by a constant times $\hn^{-1}$ uniformly in $p$.
	Using a standard scaling argument, this counterexample can be extended to any
	$\OmegaAB=(a,b)$.
\end{remark}

The approximation error estimate and the inverse inequality are extended to
higher Sobolev indices in Section~\ref{sec:sobolev}. Corresponding results for two
and more dimensions are given in Section~\ref{sec:dim}. There, also the extension
to Isogeometric Analysis is discussed.

\subsection{Known approximation error estimates}

Before proving the main theorems, we start with recalling two important pre-existing
approximation error estimates.
The first result is well-known in literature, cf.~\cite{Schumaker:1981}, Theorem~6.25 or
\cite{devore:1993}, Theorem~7.3. In the framework of Isogeometric Analysis,
such results have been used, e.g., in \cite{Bazilevs:2006}, Lemma~3.3.
\begin{theorem}\label{thrm:known}
		For each $r\in\mathbb{N}_0$, each $k\in\mathbb{N}$, each $q\in\mathbb{N}$ and each
		$p\in\mathbb{N}$, with $0\le r\le q\le p+1$ and $r\le k \le p$,
		there is a constant $C(p,k,r,q)$ such that the following approximation error estimate
		holds.
		For all $u\in H^q(\OmegaAB)$ and all grid sizes $h$, there is a spline approximation
		$\upkh \in \Spkh(\OmegaAB)$ such that
    \begin{equation*}
			    	|u-\upkh|_{H^r(\OmegaAB)} \le C(p,k,r,q) \hn^{q-r} |u|_{H^q(\OmegaAB)}
    \end{equation*}
    is satisfied.
\end{theorem}

This theorem is valid for tensor-product spaces in any dimension and 
gives a local bound for locally quasi-uniform knot vectors. However, the dependence of the constant on the 
polynomial degree has not been derived.

A major step towards estimates with explicit $p$-dependence was presented in \cite{daVeiga:2011}, 
Theorem~2, where an estimate with an explicit dependence on $p$, $k$, $r$ and $q$ was given. However,
there the continuity $k$ is limited by the upper bound $\tfrac12(p+1)$. In our notation, the theorem reads
as follows.
\begin{theorem}\label{thrm:known:2}
		There is a constant $C>0$ such that for each $r\in\mathbb{N}_0$, each $k\in\mathbb{N}$, each
		$q\in\mathbb{N}$ and each $p\in\mathbb{N}$ with $0\le r\le k\le q\le p+1$ 
		and $k \le \tfrac12(p+1)$ and all grid sizes $h$, the following approximation error estimate
		holds. For all $u\in H^q(\OmegaAB)$, there is a spline
		approximation $\upkh\in \Spkh(\OmegaAB)$ such that
    \begin{equation}\nonumber
			    	|u-\upkh|_{H^r(\OmegaAB)} \le C \hn^{q-r} (p-k+1)^{-(q-r)} |u|_{H^q(\OmegaAB)}
    \end{equation}
    is satisfied.
\end{theorem}

Again, the original result was stated for locally quasi-uniform knots. 
For any $p\geq2$ the relevant case $k=p$, which we consider, is not covered by this theorem.

Similar results to Theorem~\ref{thrm:approx} are known in approximation theory, cf.~\cite{Korneichuk:1991}.
There, however, different norms have been discussed. Hence we do not go into the details.
In~\cite{Evans:2009}, it was suggested and confirmed by numerical experiments that
Theorem~\ref{thrm:approx} is satisfied. A proof was however not given.

\subsection{Organization of this paper}

This paper is organized as follows. 
In Section~\ref{sec:prelim}, we present
the main steps of the proof of Theorem~\ref{thrm:approx} and give some preliminaries.
In the following two sections, the details of the proof are worked out.
In Section~\ref{sec:reduced}, we introduce the reduced spline space $\tiSph(\OmegaAB)$,
discuss its properties and extend Theorem~\ref{thrm:approx} to that space.
In the following section, Section~\ref{sec:inverse}, we give an inverse inequality
for $\tiSph(\OmegaAB)$. In the remainder of the paper, we generalize
those results: In Section~\ref{sec:sobolev} we consider higher
Sobolev indices and in Section~\ref{sec:dim}, the results are generalized
to two or more dimensions.

\section{Concept of the proof of Theorem~\citeThrmApprox{} and Preliminaries}\label{sec:prelim}

The proof of Theorem~\ref{thrm:approx} is based on an estimate for periodic splines, which is
formulated as Lemma~\ref{lem:approx:per}. The proof of Lemma~\ref{lem:approx:per} is
based on a telescoping argument based on a hierarchy of grids. For the proof, we require
\begin{itemize}
		\item an estimate for the difference of the spline approximations of a given function on two consecutive grids,
				cf.~\eqref{eq:whattolfa}, and
		\item an estimate for the difference between the spline approximation on
				some finest grid and the given function, cf.~Lemma~\ref{lem:non:robust}.
\end{itemize}
As the size of the finest grid approaches $0$, the constant in Lemma~\ref{lem:non:robust}
or its dependence on the spline degree $p$ does not matter, whereas the constant
in~\eqref{eq:whattolfa} directly affects the constant in the final result.

The estimate~\eqref{eq:whattolfa} is shown in Section~\ref{sec:twogrid}, cf. Lemma~\ref{lem:lfa}. There, the proof
is done by means of Fourier analysis, which causes the restriction of the analysis to equidistant
grids. The Fourier analysis follows a classical line: first, a matrix-vector formulation is
introduced, cf. Lemma~\ref{lemma:decomp}, then the symbols of the involved matrices
are derived, cf. Subsections~\ref{subsec:symbols} and~\ref{subsec:symbols2}. A closed
form for the symbol of the mass matrix is not available, so some statements on that
matrix are derived (Lemmas~\ref{lem:mass} and~\ref{lem:mass:estim}), which are used
in the proof of Lemma~\ref{lem:lfa}.

Having the result for two consecutive grids in the periodic case, we use the aforementioned 
telescoping argument to give an approximation error estimate for apprximating a general
periodic $H^1$-function. The extension to the non-periodic case is done by means of a
periodic extension.

\subsection{Periodic splines}

To establish the theory within this paper, we need to introduce spaces of periodic splines,
which we define as follows.
\begin{definition}\label{defi:Speriodic}
	Given a spline space $\Sph(\OmegaAB)$ over $\OmegaAB=(a,b)$, the \emph{periodic spline space} 
	$\Sphper(\OmegaAB)$ contains all functions $\uph\in \Sph(\OmegaAB)$ that satisfy the linear periodicity condition
	\begin{equation}\label{eq:sym:cond}
							\frac{\partial^{l}}{\partial x^{l}} \uph(a)=\frac{\partial^{l}}{\partial x^{l}} \uph(b)
							\mbox{ for all } l \in \mathbb{N}_0 \mbox{ with } l < p.
	\end{equation}
\end{definition}

The next step is to introduce a B-spline-like basis for this space. First, we introduce the
cardinal B-splines. On $\mathbb{R}$, the cardinal B-splines are defined
as follows, cf.~\cite{Schumaker:1981}, (4.22).
\begin{definition}
	 The cardinal B-splines of degree $p=0$, $\psi^{(i)}_{0}: \;\mathbb{R}\rightarrow\mathbb{R}$ coincide
	 with the characteristic function, i.e.,
  \begin{equation*}
      \psi^{(i)}_{0}(x) := \left\{
	  \begin{array}{ll}
	    1 & \mbox{ for } x \in ( i , i+1 ],\\
	    0 & \mbox{ else,}
	  \end{array}
      \right.
  \end{equation*}
  where $i \in \mathbb{Z}$. 

  The cardinal B-splines $\psi^{(i)}_{p}: \;\mathbb{R}\rightarrow\mathbb{R}$ of degree $p\in\mathbb{N}$ are
  given by the recurrence formula
  \begin{equation}\label{eq:recur:bspline}
      \psi^{(i)}_{p}(x) := \frac{x-i}{p} \psi^{(i)}_{p-1}(x) + \frac{(p+i+1)-x}{p} \psi^{(i+1)}_{p-1}(x),
  \end{equation}
  where $i \in\mathbb{Z}$.
\end{definition}
From the cardinal B-splines $\psi^{(i)}_{p}$, we derive the B-splines $\bspl^{(i)}$ on $\OmegaAB$ over 
a uniform grid of size $\hn$ by a suitable scaling and shifting. 
\begin{definition}
	 For $i\in\mathbb{Z}$ the uniform B-spline $\bspl^{(i)}: \;\OmegaAB= (a,b)\rightarrow\mathbb{R}$ of degree $p\in\mathbb{N}_0$ and 
	 grid size $\hn$ is given by
  \begin{equation}
      \bspl^{(i)}(x) := \psi^{(i)}_{p}\left(\frac{x-a}{h}\right).
  \end{equation}
\end{definition}

We obtain by construction that $\mbox{supp}(\bspl^{(i)}) \subset [i \hn+a, (i+p+1) \hn+a ]$. Hence, $-p$ and $\nh-1$ with $\nh = \hn^{-1}(b-a)$ 
are the first and last indices of the B-splines supported in $\OmegaAB$, respectively, 
i.e. $\supp (\bspl^{(i)}) \cap \OmegaAB \neq \emptyset$ is equivalent to $-p \leq i \leq \nh-1$.
Moreover, $\{\bspl^{(i)}\}^{\nh-1}_{i=-p}$ forms a basis for $\Sph$, see, e.g., \cite{Schumaker:1981}. 
Note that both $\nh$ and the basis functions depend implicitly on the choice of $\OmegaAB$, i.e., the
values $a$ and $b$. Throughout the paper, it is clear from the context which
$\OmegaAB$ is chosen.

For the construction of the basis for the periodic spline space~$\Sphper(\OmegaAB)$, we assume that
\begin{equation}\label{eq:condition-grid-size}
		hp < |\OmegaAB| = b-a,
\end{equation}
i.e., that the grid is fine enough not to have basis functions that are non-zero at both end points
of the grid, cf.~\cite{Schumaker:1981}. 
\begin{definition}\label{defi:basis-per}
For $\Sphper(\OmegaAB)$, the \emph{B-spline-like basis} $\{\bsplper^{(i)}\}^{\nh-1}_{i=0}$ is given by 
\begin{align*}
		& \bsplper^{(i)}:= \bspl^{(i)}&&\mbox{ if } i<\nh-p, \mbox{ and} \\
		& \bsplper^{(i)}:= \bspl^{(i)}+\bspl^{(i-\nh)}&&\mbox{ if } i \geq \nh-p.
\end{align*}
\end{definition}
Up to indexing, this definition coincides with~(8.6) and (8.7) in~\cite{Schumaker:1981}. 
Theorem~8.2 in~\cite{Schumaker:1981} states that~\eqref{eq:basis:varphi} is actually a basis.

As $\bspl^{(i)}$ vanishes on $\OmegaAB$ for all $i\not\in \{-p,\ldots,\nh-1\}$, we have
\begin{equation}\label{eq:basis:varphi}
		\bsplper^{(i)}=\sum_{j\in\mathbb{Z}} \bspl^{(i+j \nh)},
\end{equation}
where $\mathbb{Z}$ is the set of integers, for all $i=0,\ldots, \nh-1$. Using this definition,
we directly obtain that also $\bsplper^{(i)} = \bsplper^{(i+j\nh)}$ for any $j\in \mathbb{Z}$,
which we will use for ease of notation throughout this paper.

We call this basis B-spline-like, as each function is a non-negative linear combination of B-splines and it forms
a partition of unity on $\OmegaAB$.

\subsection{A non-robust approximation error estimate in the periodic case}

We can extend Theorem~\ref{thrm:known} for $k=p-1$ to the following Lemma \ref{lem:non:robust} stating
that the approximation error estimate is still satisfied if we approximate
periodic functions with periodic splines. First, we introduce the spaces
of periodic functions as follows.

\begin{definition}
		For $\OmegaAB=(a,b)$, the space $\widehat{H}^q(\OmegaAB)$ is the space of all
		$u\in H^q(\OmegaAB)$ that satisfy the periodicity condition
		\begin{equation}\label{eq:pc}
				\frac{\partial^{l}}{\partial x^{l}} u(a)=\frac{\partial^{l}}{\partial x^{l}} u(b)
								\mbox{ for all } l \in \mathbb{N}_0 \mbox{ with } l < q.
		\end{equation}
\end{definition}
Note that standard trace theorems guarantee that
the periodicity condition~\eqref{eq:pc} is well-defined.
For this space, the following lemma holds.
\begin{lemma}\label{lem:non:robust}
		For each $r\in\mathbb{N}_0$, each $q\in\mathbb{N}$ and each
		$p\in\mathbb{N}$ with $0\le r\le q\le p+1$, there is a constant $C(p,r,q)$
		such that the following approximation error estimate holds.
		For all $u \in \widehat{H}^{q}(\OmegaAB)$ and all grid sizes $h$, there is
		a spline approximation $\uph \in \Sphper(\OmegaAB)$ such that
    \begin{equation*}
			    	|u-\uph |_{H^r(\OmegaAB)} \le C(p,r,q) \hn^{q-r} |u|_{H^q(\OmegaAB)}
    \end{equation*}
    is satisfied.
\end{lemma}
\begin{proof}
		In the following, we assume without loss of generality that $\OmegaAB=(0,1)$. The extension
		to any other $\OmegaAB=(a,b)$, follows using a standard scaling argument.
		
		Let $w$ be the periodic extension of the function $u$ to $\mathbb{R}$, i.e., $w(x):=u(x-\lfloor x \rfloor)$.
		Note that the restriction of $w$ to any finite interval is again a function in the Sobolev space~$H^q$.
		The following of the proof is based on the proof in \S~6.4 in~\cite{Schumaker:1981}.
		We make use of the fact that the proof uses local projections. Let $Q_{p,h}: H^q(\mathbb{R})
		\rightarrow \Sph(\mathbb{R})$ be the projection operator, as introduced in (6.40) in~\cite{Schumaker:1981}.
		The value of the approximation $Q_{p,h}w$ of a function $w$ at a certain subinterval
		$I_i:=(i\;h,(i+1)\;h)\subseteq \OmegaAB$ only depends on the values of the function to be approximated in
		a certain neighborhood $\widetilde{I}_i:=((i-p)\;h,(i+p+1)\;h)$. So, from the periodicity of $w$, the
		periodicity of $Q_{p,h}w$ follows immediately. Hence its restriction to $(0,1)$ is a periodic spline,
		i.e. $Q_{p,h}w|_{(0,1)}\in\Sphper(0,1)$. We define $\uph$ to be the restriction of $Q_{p,h}w$ to $(0,1$).
		Due to \cite{Schumaker:1981}, Theorem~6.24, the local estimate
		\begin{equation}\nonumber
				|w-Q_{p,h}w|_{H^r(I_i)} \le  \widetilde{C}(p,r,q) \hn^{q-r} |w|_{H^q(\widetilde{I}_i)}.
		\end{equation}
		is satisfied for the projector $Q_{p,h}$ and
		a constant $\widetilde{C}(p,r,q)$, which is independent of~$\hn$. By summing
		over all elements, we obtain 
		\begin{align*}
				&|u-\uph |_{H^r(0,1)}^2  
					 =	|w-Q_{p,h}w |_{H^r(0,1)}^2
					  =	\sum_{i=0}^{\nh-1} |w-Q_{p,h}w|_{H^r(I_i)}^2 \\
				  & \quad \le  \widetilde{C}^2(p,r,q) \hn^{2(q-r)} \sum_{i=0}^{\nh-1} |w|_{H^q(\widetilde{I}_i)}^2
					  =\widetilde{C}^2(p,r,q) \hn^{2(q-r)} \sum_{i=0}^{\nh-1}\sum_{j=-p}^{p} |w|_{H^q(I_{i+j})}^2.
		\end{align*}
		Using the periodicity of $w$, we can express the last term using~$|u|_{H^q(I_{l})}$
		for~$l\in\{0,\ldots,\nh-~1\}$ only. By counting the occurrences of the
		summands~$|u|_{H^q(I_{l})}$, we obtain
		\begin{equation}\nonumber
						\sum_{i=0}^{\nh-1}\sum_{j=-p}^{p} |w|_{H^q(I_{i+j})}^2
						= (2p+1) \sum_{i=0}^{\nh-1} |u|_{H^q(I_{i})}^2 = (2p+1) |u|_{H^q(0,1)}^2, 
		\end{equation}
		which finishes the proof for $C(p,r,q) = (2p+1)^{1/2} \widetilde{C}(p,r,q)$.\qed
\end{proof}

\section{A robust approximation error estimate for two consecutive grids in the periodic case}\label{sec:twogrid}

In this section we analyze the case of approximating a periodic spline function on a fine grid 
by a periodic spline function on a coarser grid. In the next section,
we extend these results to the approximation of general functions and
to the non-periodic case. The extension to the non-periodic case is done by
extending functions in $H^1(0,1)$ to $(-1,1)$ by reflecting them on the $y$-axis. So,
without loss of generality, we will restrict ourselves to $\Omega=(-1,1)$ throughout this
section. Moreover, for the construction of~\eqref{eq:our:mass:formula}, we will
need that $hp<1$, which is stronger than the requirement $hp<b-a$, cf. Theorem~\ref{thrm:approx}.
So, throughout this section, we will use the following assumptions.
\begin{assumption}
		The domain is given by $\Omega=(-1,1)$ and the grid size is small enough such that
		$hp<1$ holds.
\end{assumption}

In the next section, we will make use of a telescoping argument. For this purpose,
we have to analyze a fixed interpolation operator. So, within this section, we will show 
that
\begin{equation}\label{eq:whattolfa}
    	\|(I- \widehat{\Pi}_{p,\hn} ) \uph\|_{L^2(-1,1)} \le \frac{1}{\sqrt{2}}\; \hn |\uph|_{H^1(-1,1)}
\end{equation}
holds for all $\uph\in \Sphhper(-1,1)$, where
$I$ is the identity and  $\widehat{\Pi}_{p,\hn}$  is the $H^1$-orthogonal projection operator,
given by the following definition.
\begin{definition}\label{def:H1projection}
		The projection $\widehat{\Pi}_{p,\hn}:\widehat{H}^1(-1,1)\rightarrow \Sphper(-1,1)$
		maps every $u\in \widehat{H}^1(-1,1)$ to the function $\uph \in \Sphper(-1,1)$ satisfying
		\begin{equation}\label{eq:def:projection}
				(\uph, \vph)_{\Ho(-1,1)} = (u, \vph)_{\Ho(-1,1)}
		\end{equation}
		for all $\vph\in \Sphper(-1,1)$, where
		\begin{equation*}
			(u,v)_{\Ho(-1,1)}:=(u,v)_{H^1(-1,1)} + \left(\int_{-1}^1 u(x)\dd x\right)\left(\int_{-1}^1 v(x)\dd x\right).
		\end{equation*}
\end{definition}

Within the next subsections, we will prove~\eqref{eq:whattolfa}. This will be done by
a rigorous version of Fourier analysis. Fourier analysis is a well-known tool for
analyzing convergence properties of numerical methods, cf. the work by A. Brandt, like~\cite{Brandt:1977},
and many others. It provides a framework to determine sharp bounds for the convergence
rates of multigrid methods and other iterative solvers for problems arising from partial differential equations.
This is different to classical analysis, which typically yields qualitative statements only.
For a detailed introduction into Fourier analysis, see, e.g.,~\cite{Trottenberg:2001}.
Recently, it has also been applied in the area of Isogeometric Analysis, cf.~\cite{Garoni:2014}.

Typically, Fourier analysis is done under simplifying assumptions, like assuming uniform
grids and neglecting the boundary. In this case, one refers to \emph{local} Fourier analysis
(or local mode analysis). This analysis can be understood as a heuristic method to study 
methods of interest. In a recent work, cf.~\cite{Garoni:2014}, it was understood
also as a rigorous statement for a limit case.

We, however, are interested in a completely rigorous analysis. As we restrict ourselves
to periodic spline spaces, the Fourier modes are the exact eigenvectors of the matrices
of interest, which will allow us to diagonalize these matrices using a similarity
transformation. Based on such a diagonalization, we will be able to
prove~\eqref{eq:whattolfa}.

As a first step, we introduce a matrix-vector formulation of~\eqref{eq:whattolfa}.

\subsection{A matrix-vector formulation of the estimate}

Having fixed the B-spline like basis $\{\bsplper^{(i)}\}_{i=0}^{\nh-1}$, we can write
any function $\uph\in \Sphper(-1,1)$ as a linear combination of these basis functions:
\begin{equation*}
			\uph = \sum_{i=0}^{\nh-1} \uph^{(i)} \bsplper^{(i)}.
\end{equation*}
The coefficients $\uph^{(i)}$ can be collected in a coefficient vector: We define
$\uphvec:=(\uph^{(i)})_{i=0}^{\nh-1}$. So, the vector $\uphvec$ is
the representation of the function $\uph$ with respect to the B-spline like basis.
Here and in what follows, we will always assume underlined quantities to be the
basis representation of the corresponding function with respect to the basis $\{\bsplper^{(i)}\}_{i=0}^{\nh}$.

By plugging such a decomposition into the standard $L^2$-scalar product $(\cdot,\cdot)_{L^2(-1,1)}$,
we obtain
\begin{equation*}
		(\uph,\vph)_{L^2(-1,1)} = \sum_{i=0}^{\nh-1}\sum_{j=0}^{\nh-1}\uph^{(i)}\;\vph^{(j)} 
		\; (\bsplper^{(i)},\bsplper^{(j)})_{L^2(-1,1)}.
\end{equation*}
As the grid is equidistant and the splines are periodic, we obtain that for all $i$ and $j$ the relation
$(\bsplper^{(i)},\bsplper^{(j)})_{L^2(-1,1)}=m_{p,h}^{(i-j)}$ holds with coefficients
$m_{p,h}^{(i)}:=(\bsplper^{(i)},\bsplper^{(0)})_{L^2(-1,1)}$. Those coefficients
form a circulant matrix $M_{p,h}:=(m_{p,h}^{(i-j)})_{i=0,\ldots,\nh-1}^{j=0,\ldots,\nh-1}$, which is called the
mass matrix. We immediately obtain
\begin{equation*}
		(\uph,\vph)_{L^2(-1,1)} = (\uphvec,\vphvec)_{M_{p,h}} := \vphvec^T M_{p,h}\uphvec
\end{equation*}
and
\begin{equation*}
		\|\uph\|_{L^2(-1,1)}^2 = \|\uphvec\|_{M_{p,h}}^2 := \uphvec^T M_{p,h} \uphvec.
\end{equation*}
Having a look onto the support of the functions $\bsplper^{(0)}$, we obtain that the bandwidth
of the mass matrix is $2p+1$, i.e. $m_{p,h}^{(i-j)} = 0$ for all $i,j$ with $|i-j|>p$.

Analogously to the definition of the mass matrix, we can introduce the stiffness matrix, representing the $\Ho$-scalar product. 
The stiffness matrix is given by $K_{p,h}:=(k_{p,h}^{(i-j)})_{i=0,\ldots,\nh-1}^{j=0,\ldots,\nh-1}$, where the coefficients are given by 
\begin{equation*}
		k_{p,h}^{(i)} := \left( \bsplper^{(i)}, \bsplper^{(0)} \right)_{\Ho(-1,1)}.
\end{equation*}
Since the basis functions $\bsplper^{(i)}$ form a partition of unity on $\Omega=(-1,1)$, 
$\int_{-1}^1 \bsplper^{(i)}(x) \dd x=\hn$ and further
\begin{equation}\label{eq:def:stiff}
		k_{p,h}^{(i)} = \left( \bsplper^{(i)}, \bsplper^{(0)} \right)_{H^1(-1,1)}
						+ \hn^2.
\end{equation}
Note that for uniform knot vectors the identity 
\begin{equation*}
			\frac{\partial}{\partial x} \varphi_{p,h}^{(j)}(x) = \frac{1}{\hn}\left(\varphi_{p-1,h}^{(j-1)}(x)- \varphi_{p-1,h}^{(j)}(x)\right)
\end{equation*}
holds, see e.g. (5.36) in \cite{Schumaker:1981}. This statement directly carries over to the periodic splines using relation \eqref{eq:basis:varphi}, i.e., 
\begin{equation*}
			\frac{\partial}{\partial x} \bsplper^{(j)}(x) = \frac{1}{\hn}\left(\widehat{\varphi}_{p-1,h}^{(j-1)}(x)- \widehat{\varphi}_{p-1,h}^{(j)}(x)\right)
\end{equation*}
also holds. By plugging this into~\eqref{eq:def:stiff}, the entries of
the stiffness matrix can be derived directly using the entries of the mass matrix for splines of order $p-1$. Straight-forward
calculations show that
\begin{equation}\label{eq:k:decomp}
		K_{p,h} = D_{h} M_{p-1,h} D_{h}^T + E_{h},
\end{equation}
where the gradient matrix $D_{h}:=(d_{h}^{(i-j)})_{i=0,\ldots,\nh-1}^{j=0,\ldots,\nh-1}$ is given by the coefficients
\begin{equation*}
		d_{h}^{(i)} := \frac{1}{\hn} \left\{
				\begin{array}{ll}
						1 & \mbox{ for } i\in \nh \,\mathbb{Z} \\
						-1 & \mbox{ for } i\in \nh \,\mathbb{Z}-1 \\
						0 & \mbox{ else}
				\end{array}
		\right.,
\end{equation*}
the rank-one matrix $E_h$ is given by
$E_h :=  \hn^2 \ul{\bf{1}}_{h} \ul{\bf{1}}_{h}^T$, where $\ul{\bf{1}}_{h}:=(1,\ldots,1)^T\in\mathbb{R}^{\nh}$
is a vector consisting only of ones, representing the constant function.
Note that $D_h$, $E_h$ and, consequently, $K_h$ are also circulant matrices.

To derive a matrix-vector formulation of \eqref{eq:whattolfa}, we have to introduce a matrix that represents
the canonical embedding from $\Sphper(-1,1)$ into $\Sphhper(-1,1)$.
The following lemma is rather well-known in literature, cf. \cite{Chui:1992} equation (4.3.4), and can be easily
shown by induction in $p$.
\begin{lemma}
		For all $p\in\mathbb{N}$, all grid sizes $h$ and all $x\in\mathbb{R}$,
		\begin{equation}\nonumber
				\bspl^{(j)}(x) = 2^{-p} \sum_{l=0}^{p+1} \left(\begin{array}{c}p+1\\l\end{array}\right)
																	\varphi_{p,\tfrac{h}{2}}^{(2j+l)}(x) 
		\end{equation}
		is satisfied for all $j=-p,\ldots,\nh-p-1$.
\end{lemma}
This directly carries over to the periodic splines, i.e., we obtain
\begin{equation}\label{eq:intergrid:per}
	\bsplper^{(j)}(x) = 2^{-p} \sum_{l=0}^{p+1} \left(\begin{array}{c}p+1\\l\end{array}\right)
	\widehat{\varphi}_{p,\tfrac{h}{2}}^{(2j+l)}(x)
	=  \sum_{i\in\mathbb{Z}} 
	\underbrace{2^{-p}\left(\begin{array}{c}p+1\\  i-2j  \end{array}\right)}
	_{\displaystyle p_{p,\tfrac{h}{2}}^{(i,j)}:=}
	\widehat{\varphi}_{p,\tfrac{h}{2}}^{(i)}(x).
\end{equation}
Here, we use equation \eqref{eq:basis:varphi} and that
the binomial coefficient $\left(\begin{array}{c}a\\b \end{array}\right)$ vanishes for $b\not\in \{0,\ldots,a\}$.
Again, we define the matrix $P_{p,\tfrac{h}{2}}:=(p_{p,\tfrac{h}{2}}^{(i,j)})_{i=0,\ldots, 2\nh-1}^{j=0,\ldots,\nh-1}$.
Here and in what follows, we make use of $n_{\frac{h}{2}} = 2 \nh$.

\begin{lemma}\label{lemma:decomp}
		The inequality~\eqref{eq:whattolfa} is equivalent to
		\begin{equation}\label{eq:whattolfa2}
		  		\|M_{p,\frac{\hn}{2}}^{1/2} (I-P_{p,\frac{\hn}{2}} K_{p,\hn}^{-1} P_{p,\frac{\hn}{2}}^T K_{p,\frac{\hn}{2}}) K_{p,\frac{\hn}{2}}^{-1/2}\| 
		  			\le \frac{1}{\sqrt{2}} h  ,
		\end{equation}
		which is a consequence of the combination of
					\begin{align}
			  & \|M_{p,\frac{\hn}{2}}^{1/2}M_{p-1,\frac{\hn}{2}}^{-1/2}\|\le 1\qquad \mbox{and}\label{eq:whattolfa3}\\
			  & \|M_{p-1,\frac{\hn}{2}}^{1/2} (I-P_{p,\frac{\hn}{2}} K_{p,\hn}^{-1} P_{p,\frac{\hn}{2}}^T K_{p,\frac{\hn}{2}}) K_{p,\frac{\hn}{2}}^{-1/2}\| 
			  			\le \frac{1}{\sqrt{2}}  h.  \label{eq:whattolfa4}
			\end{align}
\end{lemma}
Here and in what follows, $\|\cdot\|$ is the Euclidean norm and the square
root $A^{1/2}$ of a symmetric and positive definite matrix $A$ is that symmetric
and positive definite matrix that satisfies $A^{1/2}A^{1/2} = A$.
\begin{proof}{\em of Lemma~\ref{lemma:decomp}}
		Using the introduced matrices $K_{p,h}$ and $P_{p,\tfrac{h}{2}}$, we can rewrite~\eqref{eq:def:projection} for
		the choice $u := \uphh \in \Sphhper$ in matrix-vector form as
		\begin{equation}\nonumber
					(P_{p,\frac{\hn}{2}} \uphvec, P_{p,\frac{\hn}{2}} \vphvec)_{K_{p,\frac{\hn}{2}}}
							= (\uphhvec, P_{p,\frac{\hn}{2}}\vphvec)_{K_{p,\frac{\hn}{2}}},
		\end{equation}
		which is equivalent to 
		\begin{equation}\nonumber
			P_{p,\frac{\hn}{2}}^TK_{p,\frac{\hn}{2}}P_{p,\frac{\hn}{2}} \uphvec = 
			P_{p,\frac{\hn}{2}}^T K_{p,\frac{\hn}{2}} \uphhvec.
		\end{equation}
		This yields, using the Galerkin principle ($P_{p,\frac{\hn}{2}}^TK_{p,\frac{\hn}{2}}P_{p,\frac{\hn}{2}}=K_{p,\hn}$),
		that the coarse-grid approximation $\uphvec$ is given by
		\begin{equation}\nonumber
					\uphvec = K_{p,\hn}^{-1} P_{p,\frac{\hn}{2}}^T K_{p,\frac{\hn}{2}} \uphhvec.
		\end{equation}
		By plugging this into~\eqref{eq:whattolfa}, we see that we have to show
		\begin{equation*}
		  \|(I-P_{p,\frac{\hn}{2}} K_{p,\hn}^{-1} P_{p,\frac{\hn}{2}}^T K_{p,\frac{\hn}{2}}) \uphhvec\|_{M_{p,\frac{\hn}{2}}} 
		  			\le \frac{1}{\sqrt{2}} \hn \|\uphhvec\|_{K_{p,\frac{\hn}{2}}}
		\end{equation*}
		for all $\uphhvec \in \mathbb{R}^{2\nh}$. By rewriting this using a standard matrix norm, we obtain
		\eqref{eq:whattolfa2}.
		Using the semi-multiplicativity of matrix norms, we obtain that~\eqref{eq:whattolfa2} is
		a consequence of \eqref{eq:whattolfa3} and \eqref{eq:whattolfa4}.
\qed\end{proof}

Note that the stiffness matrix for some degree $p$
depends implicitly on the mass matrix for the degree $p-1$. So, analyzing~\eqref{eq:whattolfa4} is more convenient
than analyzing~\eqref{eq:whattolfa2} as the inequality~\eqref{eq:whattolfa4} depends just on the one mass matrix $M_{p-1,\frac{\hn}{2}}$,
whereas~\eqref{eq:whattolfa2} depends on two mass matrices: $M_{p-1,\frac{\hn}{2}}$ and $M_{p,\frac{\hn}{2}}$.
We will show~\eqref{eq:whattolfa3} in the next subsection and~\eqref{eq:whattolfa4} in the remainder of this section.

\subsection{A lemma relating the mass matrices for different polynomial degrees}

The estimate~\eqref{eq:whattolfa3} is a direct consequence of the following lemma.
\begin{lemma}\label{lem:mass}
		For all $p\in \mathbb{N}$, grid sizes $h$ and vectors
		$\ul{u}_h \in \mathbb{R}^{\nh}$, the inequality
		\begin{equation}\nonumber
					\|\ul{u}_h\|_{M_{p,h}} \le \|\ul{u}_h\|_{M_{p-1,h}}
		\end{equation}
		is satisfied. 
\end{lemma}
\begin{proof}
		First we observe that the convolution formula for cardinal B-splines, cf. equation~(13) in~\cite{Garoni:2014},
		can be carried over to the functions $\bsplper^{(i)}$, i.e., that
		\begin{equation}\label{eq:rel1}
				\bsplper^{(i)}(x) = h^{-1} \int_0^{h} \bsplperpme^{(i)}(x-t) \dd t
		\end{equation}
		holds. Let $\ul{u}_h=(u_h^{(i)})_{i=0}^{\nh-1}$. Then, using~\eqref{eq:rel1}, we have that
		\begin{align*}
				\|\ul{u}_h\|_{M_{p,h}}^2 
						& = \int_{-1}^1 \left( \sum_{i=0}^{\nh-1} u_h^{(i)} \bsplper^{(i)}(x) \right)^2 \dd x \\
						& = \int_{-1}^1 \left( \sum_{i=0}^{\nh-1} u_h^{(i)} h^{-1} \int_0^h \bsplperpme^{(i)}(x-t) \dd t \right)^2 \dd x \\
						& = h^{-2} \int_{-1}^1 \left( \int_0^{h}  \left( \sum_{i=0}^{\nh-1} u_h^{(i)}  \bsplperpme^{(i)}(x-t) \right) \dd t \right)^2 \dd x \\
						& = h^{-2} \int_{-1}^1 \left( \int_0^{h} 1 \, s(x-t) \dd t \right)^2 \dd x
		\end{align*}
		holds, where $s(x):= \sum_{i=0}^{\nh-1} u_h^{(i)} \bsplperpme^{(i)}(x-t)$.

		Now, we apply the Cauchy-Schwarz inequality to the inner integral and obtain
		\begin{align*}
				\|\ul{u}_h\|_{M_{p,h}}^2 
						& \leq h^{-2} \int_{-1}^1 \left( \int_0^{h} 1^2 \dd t \right)\left( \int_0^{h}s^2(x-t) \dd t \right) \dd x \\
						& = h^{-1} \int_{-1}^1 \int_0^{h} s^2(x-t) \dd t\,\dd x = h^{-1} \int_0^{h} \int_{-1}^1 s^2(x-t) \dd x\,\dd t.
		\end{align*}
		Observe that due to periodicity, $\int_{-1}^1 s^2(x-t) \dd x =
		\int_{-1}^1 s^2(\xi) \dd \xi$ for all $t\in[0,h]$, which implies
		\begin{align*}
				\|\ul{u}_h\|_{M_{p,h}}^2
						& \le h^{-1} \int_0^{h} \int_{-1}^1 s^2(\xi) \dd \xi\,\dd t
						 = h^{-1} \left(\int_{0}^h 1\dd t\right)\left( \int_{-1}^1  s^2(\xi) \dd \xi \right)\\
						& = \int_{-1}^1  \left( \sum_{i=0}^{\nh-1} u_h^{(i)} \bsplperpme^{(i)}(\xi) \right)^2 \dd \xi 	 = \|\ul{u}_h\|_{M_{p-1,h}}^2,
		\end{align*}
		which finishes the proof.
\qed\end{proof}

\subsection{Symbols of mass matrix and stiffness matrix}\label{subsec:symbols}

As the matrices $M_{p,h}$ and $K_{p,h}$ are circulant matrices,
we can analyze them using Fourier analysis. So, we consider the Fourier vectors
\begin{equation*}
			\ul{f}_{h,j}:=(\ee^{ 2ij h\pi  \ii})_{i=0}^{\nh-1} \qquad \mbox{ for } j = 0,\ldots, \nh-1,
\end{equation*}
where $\ii$ is the imaginary unit. 

We observe (using that the bandwith of the mass matrix is $2p+1$) that
\begin{align*}
	( M_{p,h} \ul{f}_{h,j} )_i
	 &= \sum_{l=-p}^p  m_{p,h}^{(l)} \ee^{2(i+l)jh\pi \ii }
	 = \sum_{l=-p}^p  m_{p,h}^{(l)} \ee^{2ljh\pi \ii } \ee^{2ijh\pi \ii }\\
	 &= \underbr{\sum_{l=-p}^p  m_{p,h}^{(l)} \ee^{2ljh\pi \ii  }}{ \widehat{m}_{p,h}^{(j)}:= } ( \ul{f}_{h,j} )_i
\end{align*}
for all $i=0,\ldots,\nh-1$ and $j=0,\ldots,\nh-1$ and consequently
\begin{equation}\nonumber
	M_{p,h} \ul{f}_{h,j} = \widehat{m}_{p,h}^{(j)} \ul{f}_{h,j}
\end{equation}
is satisfied for all $j=0,\ldots,\nh-1$, i.e., that $\ul{f}_{h,j}$ is an eigenvector of $M_{p,h}$ with corresponding eigenvalue
$\widehat{m}_{p,h}^{(j)}$. As we have identified $\nh$ different eigenvalues, the corresponding eigenvectors
define a basis of $\mathbb{R}^{\nh}$. Therefore, the matrix $\mathbb{F}_{h}$, obtained by collecting
the vectors $\ul{f}_{h,j}$, i.e.,
\begin{equation*}
		\mathbb{F}_{h} := \left( \begin{array}{cccc} \ul{f}_{h,0} & \ul{f}_{h,1} &  \cdots & \ul{f}_{h,\nh-1} \end{array}\right)
		= (\ee^{ 2ij h\pi  \ii})_{i=0,\ldots,\nh-1}^{j=0,\ldots,\nh-1},
\end{equation*}
is a non-singular matrix. As~$\mathbb{F}_{h}$ is the matrix built from the eigenvectors, it diagonalizes
the matrix $M_{p,h}$, i.e.,
\begin{equation}\label{eq:mhat}
		\mathbb{F}_{h}^{-1} M_{p,h} \mathbb{F}_{h} = \widehat{M}_{p,h}, 
\end{equation}
where $\widehat{M}_{p,h}:=\diag(\widehat{m}_{p,h}^{(0)},\ldots,\widehat{m}_{p,h}^{(\nh-1)})$. Analogously, we obtain
\begin{equation}\label{eq:dhat}
		\mathbb{F}_{h}^{-1} D_{h} \mathbb{F}_{h} = \widehat{D}_{h},
\end{equation}
where $\widehat{D}_{h}:=\diag(\widehat{d}_{h}^{(0)},\ldots,\widehat{d}_{h}^{(\nh-1)})$ with
\begin{equation}\label{eq:dhatcoef}
			\widehat{d}_{h}^{(j)}:=\hn^{-1}(1-\ee^{2jh\pi\ii}).
\end{equation}
Using the same construction we obtain that further 
\begin{equation}\label{eq:dhat:star}
		\mathbb{F}_{h}^{-1} D_{h}^T \mathbb{F}_{h} = \widehat{D}_{h}^*.
\end{equation}
With $\widehat{D}_{h}^*$ we denote the adjoint (the conjugate transpose) of the matrix $\widehat{D}_{h}$. Note that $E_h=\hn^2 \ul{\bf{1}}_h \ul{\bf{1}}_h^T$
is a circulant matrix with rank $1$. The only non-zero eigenvalue is $\hn$, with corresponding eigenvector
$\ul{\bf{1}}_h = \ul{f}_{h, 0}$. So, we obtain
\begin{equation}\label{eq:ehat}
			\mathbb{F}_{h}^{-1} E_h \mathbb{F}_{h} = \widehat{E}_h
\end{equation}
where $\widehat{E}_h:=\diag(\widehat{e}_{h}^{(0)},\ldots,\widehat{e}_{h}^{(\nh-1)})$ with
\begin{equation}\label{eq:ehatcoef}
	\widehat{e}_{h}^{(j)}:=\left\{\begin{array}{ll}\hn &\mbox{ for } j=0\\0&\mbox{ otherwise.}\end{array}\right.
\end{equation}
So, we can determine, $\widehat{K}_{h}$, the symbol of the stiffness matrix. Using~\eqref{eq:k:decomp},
\eqref{eq:mhat}, \eqref{eq:dhat}, \eqref{eq:dhat:star} and~\eqref{eq:ehat}, we obtain that
\begin{equation}\label{eq:khat}
		\mathbb{F}_{h}^{-1} K_{p,h} \mathbb{F}_{h}=  \widehat{K}_{h},
\end{equation}
where $\widehat{K}_{h}:=\diag(\widehat{k}_{p,h}^{(0)},\ldots,\widehat{k}_{p,h}^{(\nh-1)})$ with
\begin{equation}\label{eq:khatcoef}
		\widehat{k}_{p,h}^{(j)} := \widehat{d}_{h}^{(j)}\widehat{m}_{p-1,h}^{(j)} (\widehat{d}_{h}^{(j)})^*+\widehat{e}_{h}^{(j)}.
\end{equation}

\subsection{Symbol of the intergrid transfer}\label{subsec:symbols2}

The following lemma characterizes the symbol of the intergrid transfer.
\begin{lemma}\label{lem:phat}
		We have
		\begin{equation}\label{eq:phat}
				\mathbb{F}_{\frac{\hn}{2}}^{-1} P_{p,\frac{\hn}{2}}\mathbb{F}_{\hn} = \widehat{P}_{p,\frac{\hn}{2}},
		\end{equation}
		where $\widehat{P}_{p,\frac{\hn}{2}}:=(\widehat{p}_{p,\frac{\hn}{2}}^{(i,j)})_{i=0,\ldots,2\nh-1}^{j=0,\ldots,\nh-1}$ with
		\begin{equation}\label{eq:phatcoef}
				\widehat{p}_{p,\frac{\hn}{2}}^{(i,j)} :=
					2^{-p-1}\left\{
							\begin{array}{ll}
									\left(1+\ee^{- 2i\frac{h}{2}\pi\ii} \right)^{p+1} 
											& \mbox{ for } i - j \in \{0,\nh\}\\
									0
											& \mbox{ otherwise }
							\end{array}
					\right.
		\end{equation}
		for all $i=0,\ldots,2\nh-1$ and all $j=0,\ldots,\nh-1$.
\end{lemma}
\begin{proof}
	The equation~\eqref{eq:phat} is equivalent to $P_{\frac{\hn}{2}}\mathbb{F}_{\hn} = \mathbb{F}_{\frac{\hn}{2}} \widehat{P}_{\frac{\hn}{2}}$.
	We obtain using~\eqref{eq:intergrid:per} and the definition of $\mathbb{F}_{\hn}$ for any unit vector 
	$\ul{\textbf{I}}_{\hn}^{(j)}$ with $j=0,\ldots,\nh-1$ that
	\begin{align*}
					&P_{\frac{\hn}{2}}\mathbb{F}_{\hn}\ul{\textbf{I}}_{\hn}^{(j)}
							 = P_{\frac{\hn}{2}}\ul{f}_{\hn,j} 
							 = 2^{-p} \left( \sum_{r\in \mathbb{Z} } \left(\begin{array}{c}p+1\\ i-2r \end{array}\right) 
										\ee^{ 2jr h \pi \ii} \right)_{i=0}^{2\nh-1}.
	\end{align*}
	Because $\tfrac12(1+\ee^{ t \pi \ii})$ takes the value $0$ for $t$ being odd and $1$ for $t$ being even, we
	can substitute $r$ by $2t$ and obtain
	\begin{align*}
							&P_{\frac{\hn}{2}}\mathbb{F}_{\hn}\ul{\textbf{I}}_{\hn}^{(j)} 
							= 2^{-p-1} \left( \sum_{t\in \mathbb{Z} } \left(\begin{array}{c}p+1\\ i-t \end{array}\right) 
										\ee^{ 2 jt \frac{h}{2} \pi \ii} ( 1+\ee^{ t \pi \ii})\right)_{i=0}^{2\nh-1} \\
							&\quad = 2^{-p-1} \left( \sum_{k\in \mathbb{Z} } \left(\begin{array}{c}p+1\\ k \end{array}\right) 
										\ee^{ 2 j(i-k) \frac{h}{2} \pi \ii} ( 1+\ee^{ (i-k) \pi \ii})\right)_{i=0}^{2\nh-1} \\
							&\quad = 2^{-p-1} \sum_{k\in \mathbb{Z} } \left(\begin{array}{c}p+1\\k\end{array}\right) 
											\left( \ee^{- 2 jk \frac{h}{2}   \pi\ii} 
												\ul{f}_{\frac{\hn}{2},j} 
										+  \ee^{-2(j+\nh)k\frac{h}{2}  \pi\ii}
												 \ul{f}_{\frac{\hn}{2},j+\nh} \right) \\
							&\quad = 2^{-p-1}\left(1+\ee^{- 2j\frac{h}{2}\pi\ii} \right)^{p+1} \ul{f}_{\frac{\hn}{2},j}  
							 + 2^{-p-1}\left(1+\ee^{-2(j+\nh)\frac{h}{2}\pi\ii} \right)^{p+1} \ul{f}_{\frac{\hn}{2},j+\nh}.
	\end{align*}
	This shows that the $j$-th column of $P_{\frac{\hn}{2}}\mathbb{F}_{\hn}$ is just the combination
	of two columns of $\mathbb{F}_{\frac{\hn}{2}}$. Therefore, the matrix 
	$\widehat{P}_{\frac{\hn}{2}}$ has just two non-zero
	entries, in the $j$-th row: those which we have claimed in~\eqref{eq:phatcoef}.
\qed\end{proof}

For determining the symbol of $P_{p,\frac{h}{2}}^T$, we observe as follows.
As the Fourier modes $\ul{f}_{\hn,j}$ are pairwise orthogonal, and $\ul{f}_{\hn,j}^*\ul{f}_{\hn,j} = \nh$, we immediately
obtain $\mathbb{F}_{\hn}^*\mathbb{F}_{\hn} = \nh I$ and, consequently, $\mathbb{F}_{\hn}^{-1} = \hn \mathbb{F}_{\hn}^*$.
So, we obtain using~\eqref{eq:phat} that
\begin{equation}\label{eq:phattranspose}
			\mathbb{F}_{h}^{-1} P_{p,\frac{h}{2}}^T \mathbb{F}_{\frac{h}{2}}
						= ( \mathbb{F}_{\frac{h}{2}}^* P_{p,\frac{h}{2}} \mathbb{F}_{h}^{-*} )^*
						= ( 2 \mathbb{F}_{\frac{h}{2}}^{-1} P_{p,\frac{h}{2}} \mathbb{F}_{h} )^* 
						= 2 \widehat{P}_{p,\frac{h}{2}}^*.
\end{equation}

\subsection{Some statements on the symbol of the mass matrix}

A closed form for the symbol of the mass matrix is not known.
Within this subsection we will show a few statements characterizing the symbol,
which we will need later on.
Due to \cite{Chui:1992,Wang:2010}, we have
\begin{equation}\label{eq:our:mass:formula}
		m_{p,\hn}^{(j)} = \hn \frac{ E(2 p + 1, p + j)}{(2 p + 1)!},
\end{equation}
where $j\in\{-p,\ldots,p\}$. Here, $E(n,k)$ are the Eulerian numbers, which satisfy the recurrence relation
\begin{equation*}
		E(n, k) = (n - k) E(n - 1, k - 1) + (k + 1) E(n - 1, k)
\end{equation*}
and the initial condition
\begin{equation*}
		E(0, j) =\left\{
			\begin{array}{ll}
					1 & \mbox{for } j= 0\\
					0 & \mbox{for } j\not= 0
			\end{array}
			\right..
\end{equation*}
A similar result was also stated in~\cite{Garoni:2014}. There, the entries of the mass
matrix, i.e., the $L^2$-products of two B-splines of order $p$ have been shown to be
equal to the function value of one B-spline of order $p+1$. Using the recurrence
relation~\eqref{eq:recur:bspline}, one obtains that the result in~\cite{Garoni:2014}
is equivalent to~\eqref{eq:our:mass:formula}. 

As $m_{p,\hn}^{(j)}=m_{p,\hn}^{(-j)}$ and $\ee^{\theta \ii} +\ee^{-\theta \ii} = 2 \cos\theta$,
we obtain
\begin{equation}\nonumber
				\widehat{m}_{p,\hn}^{(j)} = \hn \sum_{l=-p}^p\frac{ E(2 p + 1, p + l)}{(2p + 1)!}\cos(2ljh\pi).
\end{equation}

The symbol is better characterized by the following lemma.
\begin{lemma}\label{lem:mass:estim}
		The following two statements hold:
		\begin{itemize}
			\item $\widehat{m}_{p,\hn}^{(j)}> 0$ for all $j=0,\ldots,\nh-1$ and
			\item	$\widehat{m}_{p,\hn}^{(j)}\le \widehat{m}_{p,\hn}^{(k)}$ for all $j,k=0,\ldots,\nh-1$ with
							$\cos(2jh\pi) \le \cos(2kh\pi)$.
		\end{itemize}
\end{lemma}
\begin{proof}
		For $c\in[0,2]$, we define
		\begin{equation*}
				g_{p}(c) := \sum_{l=-p}^p\frac{ E(2 p + 1, p + l)}{(2p + 1)!}\cos(l\arccos (c-1))
		\end{equation*}
		and observe $g_{p}(c) = \hn^{-1} \widehat{m}_{p,\hn}^{(\eta(c))}$, where $\eta(c):=\frac{1}{2h\pi}\arccos (c-1)$.
		The statement of the lemma is now equivalent to the combination of the following two statements:
		\begin{itemize}
				\item $\hn^{-1}\widehat{m}_{p,\hn}^{(\eta(0))}=g_{p}(0)>0$ and
				\item $\hn^{-1}\widehat{m}_{p,\hn}^{(\eta(c))}=g_{p}(c)$ is monotonically increasing	for $c>0$.
		\end{itemize}
		Since we can express $\cos(l\arccos (c-1))$ as the $l$-th Chebyshev polynomial,
		$g_{p}$ is a polynomial function in $c$.  Using the recurrence relation for the Eulerian numbers, we can
		derive the following recurrence formula for $g_{p}$:
		\begin{align*}
				g_{p}(c)=\frac{1+c p}{1+2 p} g_{p-1}(c)+\frac{(2-c) (1+c (2 p-1))}{p (1+2 p)} g_{p-1}'(c)+\frac{(c-2)^2 c}{p (1+2 p)} g_{p-1}''(c).
		\end{align*}
		We can make an ansatz
		\begin{equation*}
				g_{p}(c) = \sum_{j=0}^p a_{p,j} c^j,
		\end{equation*}
		where we use $0^0 = 1$, and derive the recurrence formula 
		\begin{equation*}
				a_{p,j}=\underbr{\frac{(1-j+p)^2}{p+2 p^2}}{A_{p,j}:=} a_{p-1,j-1}
						+\underbr{\frac{4j (p-j)+j+p}{p+2 p^2}}{B_{p,j}:=} a_{p-1,j}
						+\underbr{\frac{2+6 j+4 j^2}{p+2 p^2}}{C_{p,j}:=} a_{p-1,j+1}
		\end{equation*}
		for the coefficients $a_{p,j}$. 
		For $p=1$, we obtain
		\begin{equation*}
				a_{1,j} = \left\{
						\begin{array}{ll}
								\tfrac13 & \mbox{ for $j\in\{0,1\}$ } \\
								0 & \mbox{ otherwise.}
						\end{array}
				\right.
		\end{equation*}
		As $A_{p,j}>0$, $B_{p,j}>0$ and $C_{p,j}>0$ for $0\le j \le p$, one can show using
		induction in $p$ that for all $p\ge 1$:
		\begin{equation*}
				 \left\{
						\begin{array}{ll}
								a_{p,j} > 0 & \mbox{ for $j\in\{0,1,\ldots,p\}$ } \\
								a_{p,j} = 0 & \mbox{ otherwise.}
						\end{array}
				\right.
		\end{equation*}
		This immediately implies that $g_{p}(0)>0$ and that 
		$g_{p}(c)$ is monotonically increasing for $c>0$, which concludes the proof.\qed
\end{proof}

\subsection{An estimate for the projection operator}

Now, we are able to prove the following lemma.
\begin{lemma}\label{lem:lfa0}
		The inequality~\eqref{eq:whattolfa4} holds.
\end{lemma}
\begin{proof}
	The inequality~\eqref{eq:whattolfa4} is equivalent to
	\begin{equation*}
		\underbr{\hn^{-1} \|M_{p-1,\frac{h}{2}}^{1/2}(I-P_{p,\frac{h}{2}} K_{p,h}^{-1} P_{p,\frac{h}{2}}^T K_{p,\frac{h}{2}}) K_{p,\frac{h}{2}}^{-1/2}\|}{q:=}
					\le \frac{1}{\sqrt{2}} .
	\end{equation*}
	Using Galerkin orthogonality, we obtain $K_{p,h} = P_{p,\frac{h}{2}}^T K_{p,\frac{h}{2}}P_{p,\frac{h}{2}}$. Note that
	$\mathcal{H}:=I-P_{p,\frac{h}{2}} K_{p,h}^{-1} P_{p,\frac{h}{2}}^T K_{p,\frac{h}{2}}$ is a projection operator, so
	$\mathcal{H}\mathcal{H} = \mathcal{H}$. Moreover, observe that $\mathcal{H} K_{p,\frac{h}{2}}^{-1} = K_{p,\frac{h}{2}}^{-1}\mathcal{H}^T$.
	Using these identities and $\|W\|^2 = \rho(WW^T)$, where $\rho$ denotes the spectral radius,
	we obtain
	\begin{align*}\nonumber
		q^2 &= \hn^{-2}\rho( M_{p-1,\frac{h}{2}}^{-1/2}\mathcal{H} K_{p,\frac{h}{2}}^{-1}\mathcal{H}^T  M_{p-1,\frac{h}{2}}^{-1/2})
			= \hn^{-2}\rho( K_{p,\frac{h}{2}}^{-1}M_{p-1,\frac{h}{2}}^{-1}\mathcal{H} )
			 \\&= 
			\hn^{-2}\rho(
				K_{p,\frac{h}{2}}^{-1}M_{p-1,\frac{h}{2}} (I-P_{p,\frac{h}{2}} (P_{p,\frac{h}{2}}^TK_{p,\frac{h}{2}}P_{p,\frac{h}{2}})^{-1} P_{p,\frac{h}{2}}^T K_{p,\frac{h}{2}}) 
				).
	\end{align*}
	Using~\eqref{eq:mhat}, \eqref{eq:dhat}, \eqref{eq:dhat:star}, \eqref{eq:khat}, \eqref{eq:phat} and~\eqref{eq:phattranspose}, we obtain further
	\begin{align*}\nonumber
		q^2 = \hn^{-2}\rho(
				\underbr{\widehat{K}_{p,\frac{h}{2}}^{-1}\widehat{M}_{p-1,\frac{h}{2}}
				(I-2 \widehat{P}_{p,\frac{h}{2}} (2\widehat{P}_{p,\frac{h}{2}}^*\widehat{K}_{p,\frac{h}{2}}\widehat{P}_{p,\frac{h}{2}})^{-1}
					 \widehat{P}_{p,\frac{h}{2}}^* \widehat{K}_{p,\frac{h}{2}})
				}{ \widehat{T}_{p,\frac{h}{2}}:=}
				).
	\end{align*}
	Lemma~\ref{lem:mass:estim} states that all digonal entries of $\widehat{M}_{p-1,\frac{h}{2}}$ are non-zero.
	It is straight-forward to see that also the diagonal entries of
	$\widehat{K}_{p,\frac{h}{2}}$ and $\widehat{K}_{p,h}=\widehat{P}_{p,\frac{h}{2}}^*\widehat{K}_{p,\frac{h}{2}}\widehat{P}_{p,\frac{h}{2}}$
	are non-zero. So, $\widehat{T}_{p,\frac{h}{2}}$ is well-defined.
	
	Recall that
	Lemma~\ref{lem:phat} states that the matrix $\widehat{P}_{p,\frac{h}{2}}=(\widehat{p}_{p,\frac{h}{2}}^{(i,j)})_{i=0,\ldots,2\nh-1}^{j=0,\ldots,\nh-1}$
	has a block-structure, given by
	\begin{equation*}
			\widehat{p}_{p,\frac{h}{2}}^{(i,j)} = 0 \mbox{ for all } i-j \not\in\{0,\nh\}.
	\end{equation*}
	Therefore and because the matrices $\widehat{M}_{p-1,\frac{h}{2}}$ and
	$\widehat{K}_{p,\frac{h}{2}}$ are diagonal, 
	the matrix $\widehat{T}_{p,\frac{h}{2}} = (\widehat{t}_{p,\frac{h}{2}}^{(i,j)})_{i=0,\ldots,2\nh-1}^{j=0,\ldots,2\nh-1}$
	has a block-structure, given by
	\begin{equation*}
			\widehat{t}_{p,\frac{h}{2}}^{(i,j)} = 0 \mbox{ for all } i-j \not\in\{-\nh,0,\nh\}.
	\end{equation*}
	By reordering the coefficients of the matrix $\widehat{T}_{p,\frac{h}{2}}$, we obtain a block-diagonal matrix with
	blocks
	\begin{equation*}
			\mathcal{T}_{p,\frac{h}{2}}^{(l)} = \left(\begin{array}{cc}
					\widehat{t}_{p,\frac{h}{2}}^{(l,l)} & \widehat{t}_{p,\frac{h}{2}}^{(l,l+\nh)} \\
					\widehat{t}_{p,\frac{h}{2}}^{(l+\nh,l)} & \widehat{t}_{p,\frac{h}{2}}^{(l+\nh,l+\nh)}
			\end{array}\right).
	\end{equation*}
	As this block-diagonal matrix is spectrally equivalent to $\widehat{T}_{p,\frac{h}{2}}$ and the spectral radius
	of a block-diagonal matrix is just the maximum over the spectral radii of the blocks, we obtain
	\begin{equation*}
			q^2 = \rho(\widehat{T}_{p,\frac{h}{2}}) = \max_{l=0,\ldots,\nh-1} \rho( \mathcal{T}_{p,\frac{h}{2}}^{(l)} ).
	\end{equation*}
	So, in the following, we derive the spectral radius of $\mathcal{T}_{p,\frac{h}{2}}^{(l)}$ for any particular $l$. Straight-forward
	computation yields that for $l\in\{0,\ldots,\nh-1\}$, $i\in\{l,l+\nh\}$ and $j\in\{l,l+\nh\}$, we have
	\begin{equation}\label{eq:xxxy}
			\widehat{t}_{p,\frac{h}{2}}^{(i,j)} =
					\frac{\widehat{m}_{p-1,\frac{h}{2}}^{(i)}}{\widehat{k}_{p,\frac{h}{2}}^{(i)}}
					\left (\delta_{i,j} - \frac{\widehat{p}_{p,\frac{h}{2}}^{(i,l)}(\widehat{p}_{p,\frac{h}{2}}^{(j,l)})^* }{\sum_{r=0}^1
					(\widehat{p}_{p,\frac{h}{2}}^{(l+r\nh,l)})^* \widehat{k}_{p,\frac{h}{2}}^{(l+r\nh)}  \widehat{p}_{p,\frac{h}{2}}^{(l+r\nh,l)}
					} \widehat{k}_{p,\frac{h}{2}}^{(j)} \right),
	\end{equation}	
	where $\delta_{i,j}$ is the Kronecker-delta, i.e., $\delta_{i,j}=1$ for $i=j$ and $\delta_{i,j}=0$ for $i\not=j$.

	Now, consider \emph{case A}: $l\in\{1,\ldots,\nh-1\}$. Here, we plug the values of $\widehat{k}_{p,\frac{h}{2}}^{(j)}$,
	$\widehat{d}_{\frac{h}{2}}^{(j)}$, $\widehat{e}_{\frac{h}{2}}^{(j)}$ (which takes the value $0$ for $j\in\{l,l+\nh\}$),
	$\widehat{p}_{p,\frac{h}{2}}^{(i,j)}$, as given by \eqref{eq:khatcoef}, \eqref{eq:dhatcoef},
	\eqref{eq:ehatcoef} and~\eqref{eq:phatcoef}, into~\eqref{eq:xxxy} and substitute $\widehat{m}_{p-1,\frac{h}{2}}^{(l+n_h)}$ by
	$\xi \widehat{m}_{p-1,\frac{h}{2}}^{(l)}$. Doing so, the term $\widehat{m}_{p-1,\frac{h}{2}}^{(l)}$ cancels out
	and we obtain by straight-forward computation
	\begin{equation}\nonumber
			\mathcal{T}_{p,\frac{h}{2}}^{(l)} =
					\frac{1}{\delta}
					\left( 
							\begin{array}{c}
							 -z(1 - z)^{p-3} \xi
							 \\ z(1 + z)^{p-3}
							\end{array}
					\right)
					\left(
							\begin{array}{c}
									(-1)^{p}  (1 - z)^{p+1} \\  (1 + z)^{p+1}
							\end{array}
					\right)^T,
	\end{equation}	
	where $\delta:=(1 + z)^{2 p} + (-1)^p (1 - z)^{2 p} \xi$
	and $z:=\ee^{ 2 l \frac{h}{2} \pi\ii}$. Note that the computations are not a problem, as none of the symbols (except 
	$\widehat{e}_{\frac{h}{2}}^{(j)}$) takes the value $0$ for case A. Moreover, for case~A we have that $z\not\in\{-1,1\}$.
	
	Observe that $\mathcal{T}_{p,\frac{h}{2}}^{(l)}$ has rank $1$. Therefore, its spectral radius
	equals its trace, so we obtain by straight-forward computations
	that 
	\begin{align*}
			\rho( \mathcal{T}_{p,\frac{h}{2}}^{(l)} ) &=
					\frac{
								z (1+z)^{2p-2} - (-1)^p z (1-z)^{2p-2} \xi 
					}{
								(1+z)^{2p} + (-1)^p (1-z)^{2p} \xi
					} \\
					&=
					\frac{
								z^{-p+1} (1+2z+z^2)^{p-1} - (-1)^p z^{-p+1} (1-2z+z^2)^{p-1} \xi 
					}{
								z^{-p} (1+2z+z^2)^{p} + (-1)^p z^{-p} (1-2z+z^2)^{p} \xi
					} \\
					&=
					\frac{
								(z^{-1}+2+z)^{p-1} - (-1)^p (z^{-1}-2+z)^{p-1} \xi 
					}{
								(z^{-1}+2+z)^{p} + (-1)^p (z^{-1}-2+z)^{p} \xi
					} \\
					&=
					\frac{
								(2+2 c)^{p-1} - (-1)^p (-2+2 c)^{p-1} \xi 
					}{
								(2+2 c)^{p} + (-1)^p (-2+2 c)^{p} \xi
					}
					=
					\underbr{\frac{
								(1+ c)^{p-1} + (1- c)^{p-1} \xi 
					}{
							2 ( (1+ c)^{p} +  (1- c)^{p} \xi )
					}}{\Psi_p(c,\xi):=}
	\end{align*}
	holds, where $c:=\cos(2 l \frac{h}{2} \pi)$ and, as defined above,
	$\xi=\widehat{m}_{p-1,\frac{h}{2}}^{(l+\nh)}/\widehat{m}_{p-1,\frac{h}{2}}^{(l)}$. Note
	that $c\in(-1,1)$ holds as we have restricted ourselves to $l\in\{1,\ldots,\nh-1\}$.
	
	Observe that Lemma~\ref{lem:mass:estim} implies that $\xi>0$. Now, consider two cases:
	\begin{itemize}
			\item If $c=\cos(2l\frac{h}{2}\pi)> 0$, then
						$\cos(2(l+\nh)\frac{h}{2}\pi)=\cos(2l\frac{h}{2}\pi+\pi)\le 0$.
						For this case Lemma~\ref{lem:mass:estim} states that $\widehat{m}_{p-1,\frac{h}{2}}^{(l+\nh)}\le\widehat{m}_{p-1,\frac{h}{2}}^{(l)}$,
						so $\xi\le 1$ holds.
			\item Analogously, $\xi\ge 1$ holds if $c\le0$.
	\end{itemize}
	To finalize the proof of case~A, we need to show
	\begin{equation*}
			 \Psi_p\left(\cos\left(2l\frac{h}{2}\pi \right),\frac{\widehat{m}_{p-1,\frac{h}{2}}^{(l+\nh)}}{\widehat{m}_{p-1,\frac{h}{2}}^{(l)}} \right) \le \frac{1}{2}
	\end{equation*}
	for all $l=1,\ldots,\nh-1$. It suffices to show
	\begin{equation}\label{eq:cond1}
			 \Psi_p(c,\xi) \le \frac{1}{2}
	\end{equation}	
	for all $(c,\xi) \in [0,1)\times(0,1]\cup (-1,0] \times[1,\infty)$ and all $p \in \mathbb{N}$, i.e., to show the
	inequality for the whole range of $c$ and $\xi$, ignoring their dependence on $l$.
	As a next step, we observe that $\Psi_p(c,\xi)= \Psi_p(-c,\xi^{-1})$, which indicates that it suffices to
	show~\eqref{eq:cond1}
	for all $(c,\xi) \in [0,1)\times(0,1]$ and all $p \in \mathbb{N}$. We observe that
	\begin{equation}\nonumber
				\Psi_p(c,\xi) = \frac{
								1 + \left(\frac{1-c}{1+ c}\right)^{p-1} \xi 
					}{
							2 \left( (1+ c) +  (1-c)\left(\frac{1-c}{1+ c}\right)^{p-1}  \xi \right)
					}
	\end{equation}
	and
			$ \omega :=\left(\frac{1-c}{1+c}\right)^{p-1} \in [0,1]$ for $c\in[0,1]$.
	So, it suffices to show that
	\begin{equation}\label{eq:cond2}
			  \frac{
								1 + \omega \xi 
					}{
							2 ( (1+ c) +  (1-c) \omega  \xi )
					} \le \frac12
	\end{equation}
	for all $(c,\xi,\omega)\in [0,1)\times(0,1]\times[0,1]$ and all $p \in \mathbb{N}$, again
	ignoring the dependence of $\omega$ on $p$ and $c$. 

	As the denominator is always positive,~\eqref{eq:cond2} is equivalent to
	\begin{equation*}
				 1+ \omega \xi \le 1+\omega\xi + c (1-  \omega  \xi),
	\end{equation*}
	which is obviously true for all $(c,\xi,\omega)\in [0,1)\times(0,1]\times[0,1]$. 
	
	Now, we consider \emph{case B}: $l = 0$. Here, we have to use that 
	$\widehat{e}_{p,\frac{h}{2}}^{(0)}\not=0$ and obtain -- by straight-forward computation -- that
	\begin{equation*}
			\mathcal{T}_{p,\frac{h}{2}}^{(0)} = \left(
					\begin{array}{cc} 0&0\\0&\tfrac14\end{array}
			\right)
	\end{equation*}
	and consequently
	$\rho( \mathcal{T}_{p,\frac{h}{2}}^{(0)} ) = \tfrac14$. Also this is bounded from above by $\tfrac12$,
	which finishes the proof.
	\qed
\end{proof}

\subsection{The approximation error estimate}

Now, we are able to show the approximation error estimate~\eqref{eq:whattolfa}.
\begin{lemma}\label{lem:lfa}
		The inequality~\eqref{eq:whattolfa}, i.e.,
		\begin{equation}\nonumber
		    	\|(I- \widehat{\Pi}_{p,\hn} ) \uphh\|_{L^2(-1,1)} \le \frac{1}{\sqrt{2}}\; \hn |\uphh|_{H^1(-1,1)},
		\end{equation}
		holds for all $\uphh\in \Sphhper(-1,1)$.
\end{lemma}
\begin{proof}
	Lemma~\ref{lemma:decomp} states that~\eqref{eq:whattolfa} is a consequence of
	\eqref{eq:whattolfa3} and \eqref{eq:whattolfa4}. As 
	Lemma~\ref{lem:mass} shows~\eqref{eq:whattolfa3} and
	Lemma~\ref{lem:lfa0} shows~\eqref{eq:whattolfa4}, this finishes the proof.\qed
\end{proof}

\section{The proof of Theorem~\citeThrmApprox{}}\label{sec:thrm1}

In the previous section, we have given a proof for the approximation error of discretized
functions between two consecutive grids. Using a telescoping argument, 
we can extend this result to an approximation error estimate for general functions. As in the last
section, we first consider the periodic case.
\begin{lemma}\label{lem:approx:per}
    For all $u \in \widehat{H}^{1}(-1,1)$, all grid sizes $h$ and each $p\in \mathbb{N}$, with $hp<1$, 
    \begin{equation*}
	    \|(I-\widehat{\Pi}_{p,h})u\|_{L^2(-1,1)} \le \sqrt{2}\; \hn |u|_{H^1(-1,1)}
    \end{equation*}
    is satisfied, where $\widehat{\Pi}_{p,h}$ is given as in Definition~\ref{def:H1projection}.
\end{lemma}
\begin{proof}	
  	Using a telescoping argument, i.e. iteratively applying the triangular inequality, and the relation $\widehat{\Pi}_{p,2h} \widehat{\Pi}_{p,h} = \widehat{\Pi}_{p,2h}$ for the projectors, we obtain for any $q\in\mathbb{N}$
  	\begin{align*}
	    \|(I-\widehat{\Pi}_{p,h})u\|_{L^2(-1,1)} & \le \|(I-\widehat{\Pi}_{p,2^{-q}h})u \|_{L^2(-1,1)} \\
	    				&\qquad+ \sum_{l=0}^{q-1} \|(I-\widehat{\Pi}_{p,2^{-l}h} )\widehat{\Pi}_{p,2^{-l-1}h} u \|_{L^2(-1,1)}.
    \end{align*}
    We use Lemma~\ref{lem:non:robust} and a standard Aubin-Nitsche duality argument
    to estimate $\|(I-\widehat{\Pi}_{p,2^{-q}h})u\|_{L^2(-1,1)}$ from above.
    Using \cite{Braess:1997}, Lemma~7.6, and Lemma~\ref{lem:non:robust} for $r=1$ and $q=2$, we immediately
    obtain
    \begin{equation}\label{eq:aubin} 
    		\|(I-\widehat{\Pi}_{p,2^{-q}h})u\|_{L^2(-1,1)} \le \widetilde{C}(p) 2^{-q} h \|u\|_{H^1(-1,1)},
    \end{equation}
    where $\widetilde{C}(p)$ is independent of the grid size. 
  	Using~\eqref{eq:aubin} and Lemma~\ref{lem:lfa}, we obtain
  	\begin{align*}
	    \|(I-\widehat{\Pi}_{p,h})u \|_{L^2(-1,1)} &\le \widetilde{C}(p) \; 2^{-q}\hn  \|u \|_{H^1(-1,1)} \\
	    							&\qquad + \sum_{l=0}^{q-1}	\frac{1}{\sqrt{2}} \;2^{-l} \hn  |\widehat{\Pi}_{p,2^{-l-1}h} u|_{H^1(-1,1)}.
    \end{align*}
    Because $\widehat{\Pi}_{p,h}$ is $H^1$-orthogonal, we obtain 
    $|\widehat{\Pi}_{p,2^{-l-1}h} u|_{H^1(-1,1)} \leq |u |_{H^1(-1,1)}$ and further
    \begin{equation*}
	    \|(I-\widehat{\Pi}_{p,h})u \|_{L^2(-1,1)}  \le \widetilde{C}(p) \; 2^{-q} \hn  \|u \|_{H^1(-1,1)} + \sum_{l=0}^{q-1} 
	    																						\frac{1}{\sqrt{2}} \;2^{-l} \hn  |u |_{H^1(-1,1)}.
    \end{equation*}
  	The summation formula for the infinite geometric series gives
  	\begin{equation*}
	    \|(I-\widehat{\Pi}_{p,h})u \|_{L^2(-1,1)}  \le \widetilde{C}(p) \; 2^{-q} \hn  \|u \|_{H^1(-1,1)} + 
	    																						2 \frac{1}{\sqrt{2}} \hn |u|_{H^1(-1,1)}.
    \end{equation*}
    As this is true for all $q\in \mathbb{N}$, we can take the limit $q\rightarrow \infty$ and obtain the desired
    result.\qed
\end{proof}

Having this result, we note that
Theorem~\ref{thrm:approx} is just the extension of Lemma~\ref{lem:approx:per} to
the non-periodic case. So, we can easily prove Theorem~\ref{thrm:approx}.
\begin{proof}{\em of Theorem~\ref{thrm:approx}}
		In the following, we assume without loss of generality that $\OmegaAB=(0,1)$. The extension
		to any other $\OmegaAB=(a,b)$, follows using a standard scaling argument.

		Observe that any $u\in H^{1}(0,1)$ can be extended to a $w\in \widehat{H}^1(-1,1)$
		by defining $w(x):=u(|x|)$. The assumption $hp<1$ in Theorem~\ref{thrm:approx} guarantees that 
		Lemma~\ref{lem:approx:per} can be applied. 
		We set $\wph:= \widehat{\Pi}_{p,h} w \in \Sphper(-1,1)$ as in 
		Lemma~\ref{lem:approx:per}, such that 
    \begin{equation*}
	    \|w-\wph\|_{L^2(-1,1)} \le \sqrt{2}\; \hn |w|_{H^1(-1,1)}.
    \end{equation*}
    The function $\wph$ is symmetric, i.e., $\wph(x)=\wph(-x)$ holds. This can be seen by
    the following argument: As $w$ satisfies $w(x)=w(-x)$, we have for $\widetilde{w}_{p,h}(x):=\wph(-x)$
    \begin{equation*}
    		\|w-\wph\|_{L^2(-1,1)} = \|w-\widetilde{w}_{p,h}\|_{L^2(-1,1)}
    \end{equation*}
		and as $\wph$ was a unique minimizer, consequently $\wph(x) = \widetilde{w}_{p,h}(x)=\wph(-x)$ holds.
    By restricting $\wph$ to $(0,1)$, we obtain a function $\uph\in \Sph(0,1)$. 
    This function satisfies the desired approximation error estimate since  
    $|w|_{H^1(-1,1)} = \sqrt{2} |u|_{H^1(0,1)}$ and 
    $\|w-\wph\|_{L^2(-1,1)} = \sqrt{2} \|u-~\uph\|_{L^2(0,1)}$ hold due to
    the symmetry of $w$.\qed
\end{proof}

\section{Approximation error estimate for a reduced spline space}\label{sec:reduced}

In the proof of Theorem~\ref{thrm:approx} we have defined $\uph$ to be the restriction of a symmetric and
periodic spline $\wph \in \Sphper(-1,1)$ to $(0,1)$. So, we know more about $\uph$ than just $\uph\in\Sph(0,1)$.
Throughout this Section we again assume $hp<|\Omega|$.

As we have shown in the proof of Theorem~\ref{thrm:approx} the spline~$\wph$ is symmetric, 
i.e., $\wph(x)=\wph(-x)$, so we have
\begin{equation*}
			\frac{\partial^{l}}{\partial x^{l}} \wph(x)= (-1)^l \frac{\partial^{l}}{\partial x^{l}} \wph(-x)
			\mbox{ for all } l \in \mathbb{N}_0 .
\end{equation*}
By plugging~$x=0$ into this relation, we obtain that all odd derivatives vanish
for~$x=0$. By plugging~$x=1$ into the relation, we obtain together with~\eqref{eq:sym:cond}
that also for~$x=1$ all odd derivatives vanish.

So, we have shown that the approximation error estimate~\eqref{eq:thrm:approx} is still satisfied
if we restrict the approximating spline $\uph$ to be in the space $\tiSph(0,1)$, defined
as follows.

\begin{definition}\label{defi:Ssymm}
		Given a spline space $\Sph(\OmegaAB)$ over $\OmegaAB=(a,b)$, the \emph{space of splines 
		with vanishing odd derivatives} $\tiSph(\OmegaAB)$ is the space of all $\uph\in \Sph(\OmegaAB)$
		that satisfy the following condition:
		\begin{equation*}
						\frac{\partial^{2l+1}}{\partial x^{2l+1}} \uph(a)=\frac{\partial^{2l+1}}{\partial x^{2l+1}} \uph(b) = 0
						\mbox{ for all } l \in \mathbb{N}_0 \mbox{ with } 2l+1 < p.
		\end{equation*}
\end{definition}

Using a standard scaling argument, we can again extend the result for $\OmegaAB=(0,1)$
to any $\OmegaAB=(a,b)$ and obtain the following Corollary.

\begin{corollary}\label{cor:approx:nonper}
    For all $u\in H^1(\OmegaAB)$, all grid sizes $h$ and all $p\in\mathbb{N}$, with $hp<|\Omega|$,
    there is a spline approximation $\uph\in \tiSph(\OmegaAB)$ such that
    \begin{equation*}
			    	\|u-\uph\|_{L^2(\OmegaAB)} \le \sqrt{2}\; \hn |u|_{H^1(\OmegaAB)}
    \end{equation*}
    is satisfied.
\end{corollary}

In the Appendix, we will introduce a basis for the space~$\tiSph(\OmegaAB)$. Based on the bases
of those spaces, we obtain that their dimensions are as given in Table~\ref{tab:dof}.
\begin{table}
	\begin{center}
		\begin{tabular}{cccc}
				\hline\noalign{\smallskip}
										 & dim $\Sph(0,1)$ & dim $\Sphper(0,1)$ & dim $\tiSph(0,1)$ \\
			  \noalign{\smallskip}\hline\noalign{\smallskip}
				$p$ even & $n+p$               & $n$                       & $n$  \\
				$p$ odd  & $n+p$               & $n$                       & $n+1$    \\
				\noalign{\smallskip}\hline
		\end{tabular}
		\caption{Degrees of freedom, where $n$ is the number of elements in $(0,1)$.}
		\label{tab:dof}
	\end{center}
\end{table}

\section{An inverse inequality for the reduced spline space}\label{sec:inverse}

For the space~$\tiSph(\OmegaAB)$, a robust inverse inequality holds. Note that an
extension to~$\Sph(\OmegaAB)$ is not possible (cf. Remark~\ref{rem:counterexample}).

\begin{theorem}\label{thrm:inverse}
	For all grid sizes $h$ and each $p\in \mathbb{N}$,
	\begin{equation}\label{eq:inv2}
		|\uph|_{H^1(\OmegaAB)} \le 2 \sqrt{3} \hn^{-1} \|\uph\|_{L^2(\OmegaAB)}
	\end{equation}
	is satisfied for all $\uph\in \tiSph(\OmegaAB)$.
\end{theorem}
\begin{proof}
		In the following, we assume without loss of generality that $\OmegaAB=(0,1)$. The extension
		to any other $\OmegaAB=(a,b)$, follows directly using a standard scaling argument.
		We can extend every $\uph\in\tiSph(0,1)$ to $(-1,1)$ by defining $\wph(x):=\uph(|x|)$
		and obtain $\wph \in \Sphper(-1,1)$. \eqref{eq:inv2} is equivalent to
		\begin{equation}\label{eq:inv2aa}
			|\wph|_{H^1(-1,1)} \le 2 \sqrt{3} \hn^{-1} \|\wph\|_{L^2(-1,1)}.
		\end{equation}
		This is shown	using induction in $p$ for all $u\in \tiSph(-1,1)$.
		For $p=1$, \eqref{eq:inv2aa} is known, cf.~\cite{schwab:1998}, Theorem~3.91.
		
		Now, we show that the constant does not increase for larger $p$. So assume $p>1$ to be fixed.
		Due to the periodicity and due to the Cauchy-Schwarz inequality,
		\begin{align*}
		        |\wph|_{H^1(-1,1)}^2 &= \int_{-1}^1 (\wph')^2 dx = -\int_{-1}^1 \wph'' \wph dx \\
													&\le \|\wph''\|_{L^2(-1,1)} \|\wph\|_{L^2(-1,1)}= |\wph'|_{H^1(-1,1)} \|\wph\|_{L^2(-1,1)}
		\end{align*}
		is satisfied.
		Using the induction assumption (and~$\wph'\in \widehat{S}_{p-1,h}(-1,1)$, cf. \cite{Schumaker:1981}, Theorem~5.9), we know that
		\begin{equation*}
				|\wph'|_{H^1(-1,1)} \le 2 \sqrt{3} \hn^{-1} \|\wph'\|_{L^2(-1,1)} = 2 \sqrt{3} \hn^{-1} |\wph|_{H^1(-1,1)}.
		\end{equation*}
		Combining these results, we obtain
		\begin{equation*}
		        |\wph|_{H^1(-1,1)}^2 \le  2 \sqrt{3} \hn^{-1} |\wph|_{H^1(-1,1)}\|\wph\|_{L^2(-1,1)}
		\end{equation*}
		and further
		\begin{equation*}
		        |\wph|_{H^1(-1,1)} \le  2 \sqrt{3} \hn^{-1} \|\wph\|_{L^2(-1,1)}.
		\end{equation*}
		This shows \eqref{eq:inv2aa}, which concludes the proof.\qed
\end{proof}

\begin{remark}
		Neither Theorem~3.91 in~\cite{schwab:1998}, nor any of the arguments in the proof
		of Theorem~\ref{thrm:inverse} requires the grid to be equidistant. So, also having a
		general grid, estimate 
		\begin{equation}\nonumber
				|\uph|_{H^1(\OmegaAB)} \le 2 \sqrt{3} \; h_{\min}^{-1} \|\uph\|_{L^2(\OmegaAB)}
		\end{equation}
		is satisfied for all splines $\uph$ on $\OmegaAB=(a,b)$ with vanishing odd derivatives at
		the boundary. Here, as in any standard inverse inequality, $h_{\min}$ is the size of
		the \emph{smallest} element.
\end{remark}

As we have proven both an approximation error estimate and a corresponding inverse inequality,
both of them are sharp (up to constants independent of $p$ and $\hn$). First, we show that
there is a lower bound for the approximation error. As~\eqref{eq:corr:sharp1} is obviously
true for constant functions, we show that there also exist other functions satisfying this inequality.

\begin{corollary}\label{corr:sharp1}
    For all grid sizes $h$ and each $p\in\mathbb{N}$,
    there is a non-constant function $u\in H^1(\OmegaAB)$ such that
    \begin{equation}\label{eq:corr:sharp1}
			    	\inf_{\uph\in \tiSph(\OmegaAB)} \|u-\uph\|_{L^2(\OmegaAB)} \ge \frac{1}{4\sqrt{3}}\; \hn |u|_{H^1(\OmegaAB)}.
    \end{equation}
\end{corollary}
\begin{proof}
		Let $u\in S_{p,\frac{\hn}{2}}(\OmegaAB)\backslash\{0\}$ be such that $(u,\tilde{u}_{p,\hn})_{L^2(\OmegaAB)}=0$ for
		all $\tilde{u}_{p,\hn}\in \Sph(\OmegaAB)$. As the constant functions are in $\Sph(\OmegaAB)$, this
		orthogonality implies that $u$ is non-constant.
		Using this orthogonality, we know that the infimum in~\eqref{eq:corr:sharp1} is taken for $\uph=0$.
		So, we obtain using Theorem~\ref{thrm:inverse} $\inf_{\uph\in \tiSph(\OmegaAB)} \|u-\uph\|_{L^2(\OmegaAB)}
		= \|u\|_{L^2(\OmegaAB)} \ge \frac{1}{2\sqrt{3}}\; \frac{\hn}{2} |u|_{H^1(\OmegaAB)}$,
		which finishes the proof.
\qed\end{proof}
Similarly, we can give an lower bound for the inverse inequality. Again, we show the existence of a non-trivial function.
\begin{corollary}\label{corr:sharp2}
    For all grid sizes $h$ with $2hp<|\OmegaAB|$ and each $p\in\mathbb{N}$,
    there is a non-constant function $\uph\in \tiSph(\OmegaAB)$ such that
    \begin{equation}\nonumber
			    |\uph|_{H^1(\OmegaAB)} \ge \frac{1}{2\sqrt{2}}\; \hn^{-1} \|\uph\|_{L^2(\OmegaAB)}.
    \end{equation}
\end{corollary}
\begin{proof}
		Let $\uph\in \Sph(\OmegaAB)\backslash\{0\}$ be such that $(\uph,\tilde{u}_{p,2\hn})_{L^2(\OmegaAB)}=0$ for
		all $\tilde{u}_{p,2\hn}\in S_{p,2\hn}(\OmegaAB)$. As the constant functions are in $S_{p,2\hn}(\OmegaAB)$,
		this orthogonality implies that $\uph$ is non-constant.
		Using this orthogonality and Theorem~\ref{thrm:approx}, we obtain
		$\|\uph\|_{L^2(\OmegaAB)} = \inf_{u_{p,2\hn}\in \widetilde{S}_{p,n-1}(\OmegaAB)} \|\uph-u_{p,2\hn}\|_{L^2(\OmegaAB)} \le
		\sqrt{2} (2\hn) |\uph|_{H^1(\OmegaAB)}$, which finishes the proof.
\qed\end{proof}

\section{An extension to higher Sobolev indices}\label{sec:sobolev}

We can easily lift the statement of Theorem~\ref{thrm:approx}
(and also Corollary~\ref{cor:approx:nonper}) up to higher Sobolev indices.

\begin{theorem}\label{thrm:approx:sob}
    For all grid sizes $h$, each $q \in \mathbb{N}$ and each $p\in \mathbb{N}$
		with $0< q\le p+1$ and with $h(p-q+1)<|\OmegaAB|$,
		there is for each $u\in H^q(\OmegaAB)$,
    a spline approximation $\uph\in \tiSph^{(q)}(\OmegaAB)$ such that
    \begin{equation*}
	    |u-\uph|_{H^{q-1}(\OmegaAB)} \le  \sqrt{2} \; \hn |u|_{H^q(\OmegaAB)},
    \end{equation*}
    where $\tiSph^{(q)}(\OmegaAB)$ is the space of all $\uph \in \Sph(\OmegaAB)$ that
    satisfy the following symmetry condition:
    \begin{equation*}
    				\frac{\partial^{2l+q}}{\partial x^{2l+q}} \uph(a)=\frac{\partial^{2l+q}}{\partial x^{2l+q}} \uph(b) = 0
						\mbox{ for all } l \in \mathbb{N}_0 \mbox{ with } 2l+q < p.
    \end{equation*}
\end{theorem}
\begin{proof}
		Let again $\OmegaAB=(0,1)$ without loss of generality. 
		The proof is done by induction. From Corollary~\ref{cor:approx:nonper}, we know the estimate for $q=1$
		(as $\tiSph^{(1)}(0,1)=\tiSph(0,1)$) and all $p> q-1=0$. For $q=1$ and $p=q-1=0$, the estimate
		is a well-known result, cf. \cite{Schumaker:1981}, Theorem~6.1,~(6.7), where (in our notation)
		$|u-u_{0,h}|_{L^2(0,1)} \le \hn |u|_{H^1(0,1)}$ has been shown.
		
		So, now we assume to know the estimate for some $q-1$ and show it for $q$.
		
		As $u\in H^q(0,1)$, we know that $u'\in H^{q-1}(0,1)$, so we can apply the induction hypothesis and
		obtain that there is some $u_{p-1,n}\in \widetilde{S}_{p-1,n}^{(q-1)}(0,1)$ with
		\begin{equation*}
	    |u'-u_{p-1,n}|_{H^{q-2}(0,1)} \le  \sqrt{2} \; \hn |u'|_{H^{q-1}(0,1)}.
    \end{equation*}
    Define 
    \begin{equation}\label{eq:thrm:approx:sob}
    		\uph(x):=c+\int_0^x u_{p-1,n}(\xi)d\xi.
    \end{equation}
    Note that $\uph\in \Sph(0,1)$ as integrating increases 
    both the polynomial degree and the differentiability by $1$, cf.~\cite{Schumaker:1981}, Theorem~5.16.
    After integrating, the boundary conditions on the $l$-th derivative
    become conditions on the $l+1$-st derivative, therefore we further have $\uph\in \tiSph^{(q)}(0,1)$.

    Therefore, we have
    \begin{equation*}
	    |u'-\uph'|_{H^{q-2}(0,1)} \le  \sqrt{2} \; \hn |u'|_{H^{q-1}(0,1)},
    \end{equation*}
    which is the same as
    \begin{equation*}
	    |u-\uph|_{H^{q-1}(0,1)} \le  \sqrt{2} \; \hn |u|_{H^{q}(0,1)}.
    \end{equation*}
	The bound on the grid size with respect to the degree, i.e. $h(p-q+1)<|\OmegaAB|$ is sufficient, as the degree of $\partial^{q-1}/\partial x^{q-1} u$ is equal to $p-q+1$. 
   	This finishes the proof.\qed
\end{proof}
\begin{remark}
	The integration constant (integration constants for $q>2$) in~\eqref{eq:thrm:approx:sob} can be used
	to guarantee that
	\begin{equation*}
		\int_{\OmegaAB} \frac{\partial^l}{\partial x^l}(u(x)-\uph(x)) \dd x= 0
	\end{equation*}
	for all $l\in\{0,1,\ldots,q-1\}$.
\end{remark}

For the spaces $\tiSph^{(q)}(\OmegaAB)$ there is again an inverse inequality.
\begin{theorem}\label{thrm:inverse:sob}
	For all grid sizes $h$, each $q\in \mathbb{N}$ and each $p\in \mathbb{N}$ with $0< q \le p+1$,
	\begin{equation}\label{eq:inv2:sob}
		|\uph|_{H^q(\OmegaAB)} \le 2 \sqrt{3} \hn^{-1} |\uph|_{H^{q-1}(\OmegaAB)}
	\end{equation}
	is satisfied for all $\uph\in \tiSph^{(q)}(\OmegaAB)$, where
	$\tiSph^{(q)}(\OmegaAB)$ is as defined in Theorem~\ref{thrm:approx:sob}.
\end{theorem}
\begin{proof}
	First note that~\eqref{eq:inv2:sob} is equivalent to
	\begin{equation}\label{eq:inv2:sob:2}
		\left|\frac{\partial^{q-1}}{\partial x^{q-1}} \uph\right|_{H^1(\OmegaAB)} \le 2 \sqrt{3} \hn^{-1} 
				\left\|\frac{\partial^{q-1}}{\partial x^{q-1}} \uph\right\|_{L^2(\OmegaAB)}.
	\end{equation}
	As $\frac{\partial^{q-1}}{\partial x^{q-1}}\uph\in \widetilde{S}_{p-q+1,n}^{(1)}(\OmegaAB) = \widetilde{S}_{p-q+1,n}(\OmegaAB)$,
	cf. \cite{Schumaker:1981}, Theorem~5.9,
	the estimate~\eqref{eq:inv2:sob:2} follows directly from Theorem~\ref{thrm:inverse}.\qed
\end{proof}

Again, as we have both an approximation error estimate and an inverse inequality, we know that
both of them are sharp (cf. Corollaries~\ref{corr:sharp1} and~\ref{corr:sharp2}).


The following theorem is directly obtained from telescoping.

\begin{theorem}\label{thrm:approx:sob:2}
		For all grid sizes $h$, each $q\in\mathbb{N}_0$, each $p\in\mathbb{N}$, each $r\in\mathbb{N}$
		with $0\le r\le q\le p+1$ and $h(p-r)<|\OmegaAB|$, there is for each $u\in H^q(\OmegaAB)$
    a spline approximation $\uph\in \Sph(\OmegaAB)$ such that
    \begin{equation*}
			    	|u-\uph|_{H^r(\OmegaAB)} \le (\sqrt{2}\; \hn)^{q-r} |u|_{H^q(\OmegaAB)}
    \end{equation*}
    is satisfied.
\end{theorem}
\begin{proof}
	Theorem~\ref{thrm:approx:sob} states the desired result for $r=q-1$. For $r<q-1$, the
	statement is shown by induction in $r$. So, we assume to know
	the desired result for some $r$, i.e., there is a spline approximation $\wph\in \Sph(\OmegaAB)$ such that
    \begin{equation}\label{eq:thrm:0:ia}
				|u-\wph|_{H^r(\OmegaAB)} \le (\sqrt{2}\; \hn)^{q-r} |u|_{H^q(\OmegaAB)}.
    \end{equation}
    Now, we show that there is some $\uph\in \Sph(\OmegaAB)$ such that
    \begin{equation}\label{eq:thrm:0:ih}
				|u-\uph|_{H^{r-1}(\OmegaAB)} \le (\sqrt{2}\; \hn)^{q-(r-1)} |u|_{H^q(\OmegaAB)}.
    \end{equation}
    As $u-\wph\in H^r(\OmegaAB)$,
    Theorem~\ref{thrm:approx:sob} states that there is a function $\uph\in \Sph(0,1)$ such
    that
    \begin{equation*}
				|u-\uph|_{H^{r-1}(\OmegaAB)} \le \sqrt{2}\; \hn |u-\wph|_{H^r(\OmegaAB)},
    \end{equation*}
    which shows together with the induction assumption~\eqref{eq:thrm:0:ia} the induction
    hypothesis~\eqref{eq:thrm:0:ih}. Again, the bound on the grid size $h(p-r)<|\OmegaAB|$ follows directly from the bounds in Theorem~\ref{thrm:approx:sob}. \qed
\end{proof}

Here, it is not known to the authors how to choose a proper subspace of~$\Sph(\OmegaAB)$
such that a complementary inverse inequality can be shown.

\section{Extension to two and more dimensions and application in Isogeometric Analysis}\label{sec:dim}

Without loss of generality and to simplify the notation, we restrict ourselves to $\Omega:=(0,1)^d$ throughout this section. 
We can extend Theorem~\ref{thrm:approx} (and also Corollary~\ref{cor:approx:nonper}) to the following theorem for a tensor-product
structured grid on $\Omega$. 
Here, we can introduce $\widetilde{W}_{p,h}(\Omega) = \otimes_{l=1}^d \tiSph(0,1)$. Let $n=\nh$, for even $p$, and $n=\nh+1$ for odd $p$. 
Assuming that $(\bsplti^{(0)},\ldots , \bsplti^{(n-1)})$ is a basis of $\tiSph(0,1)$, the space
$\widetilde{W}_{p,h}(\Omega)$ is given by
\begin{equation*}
		\widetilde{W}_{p,h}(\Omega)=\left\{w\,:\,w(x_1,\ldots,x_d)=\hspace{-2mm}\sum_{i_1,\ldots,i_d=0}^{n-1}\hspace{-2mm} w_{i_1,\ldots,i_d} \bsplti^{(i_1)}(x_1) \cdots \bsplti^{(i_d)}(x_d) \right\}.
\end{equation*}
\begin{theorem}\label{eq:approx2d}
    For all $u\in H^1(\Omega)$, all grid sizes $h$ and each $p\in\mathbb{N}_0$, with $hp<1$, 
    there is a spline approximation $\wph\in \widetilde{W}_{p,n}(\Omega)$ such that
    \begin{equation*}
			    	\|u-\wph\|_{L^2(\Omega)} \le \sqrt{2d}\; \hn |u|_{H^1(\Omega)}
    \end{equation*}
    is satisfied.
\end{theorem}
The proof is similar to the proof in~\cite{Beirao:2012}, Section 4, for the two dimensional case. To
keep the paper self-contained we give a proof of this theorem.
\begin{proof}{\em of Theorem~\ref{eq:approx2d}}
		For sake of simplicity, we restrict ourselves to $d=2$. The extension to more dimensions
		is completely analogous. Here
		\begin{equation*}
				\widetilde{W}_{p,h}(\Omega)=\tiSph(0,1)\otimes\tiSph(0,1)=\left\{w\;:\;w(x,y)=\sum_{i,j=0}^{n-1} w_{i,j} \bsplti^{(i)}(x) \bsplti^{(j)}(y) \right\}.
		\end{equation*}

		We assume $u\in C^\infty(\Omega)$ and show the desired result using a standard
		density argument.
		Using Corollary~\ref{cor:approx:nonper}, we can introduce for each $x\in(0,1)$ a function
		$v(x,\cdot)\in \tiSph(0,1)$ with
		\begin{equation*}
				\|u(x,\cdot)-v(x,\cdot)\|_{L^2(0,1)} \le \sqrt{2}\; \hn |u(x,\cdot)|_{H^1(0,1)}.
		\end{equation*}
		By squaring and taking the integral over $x$, we obtain
		\begin{equation}\label{eq:2d:1}
				\|u-v\|_{L^2(\Omega)} \le \sqrt{2}\; \hn \left\|\frac{\partial}{\partial y} u\right\|_{L^2(\Omega)}.
		\end{equation}
		By choosing $v(x,\cdot)$ to be the $L^2$-orthogonal projection, we also have
		\begin{equation*}
				\|v(x,\cdot)\|_{L^2(0,1)} \le \|u(x,\cdot)\|_{L^2(0,1)}
		\end{equation*}
		for all $x\in(0,1)$ and consequently
		\begin{equation}\label{eq:2d:x}
				\left\|\frac{\partial}{\partial x}v(x,\cdot)\right\|_{L^2(0,1)} \le \left\|\frac{\partial}{\partial x}u(x,\cdot)\right\|_{L^2(0,1)}.
		\end{equation}
		As $v(x,\cdot)\in \tiSph(0,1)$, there are coefficients $v_j(x)$ such that
		\begin{equation*}
				v(x,y) = \sum_{j=0}^{n-1} v_j(x) \bsplti^{(j)}(y).
		\end{equation*}
		Using Corollary~\ref{cor:approx:nonper}, we can introduce for each $j\in\{0,\ldots,N\}$ a function
		$w_j\in \tiSph(0,1)$ with
		\begin{equation}\label{eq:2d:2a}
				\|v_j-w_j\|_{L^2(0,1)} \le \sqrt{2}\; \hn |v_j|_{H^1(0,1)}.
		\end{equation}
		Next, we introduce a function $w$ by defining
		\begin{equation*}
			w(x,y):=\sum_{j=0}^{n-1} w_{j}(x) \bsplti^{(j)}(y),
		\end{equation*}
		which is obviously a member of the space $\widetilde{W}_{p,n}(\Omega)$.
		By squaring \eqref{eq:2d:2a}, multiplying it with $\bsplti^{(j)}(y)^2$, summing over $j$ and
		taking the integral, we obtain
		\begin{equation}\nonumber
				\int_0^1 \sum_{j=0}^{n-1}\|v_j-w_j\|_{L^2(0,1)}^2\bsplti^{(j)}(y)^2\dd y
						 \le 2\; \hn^2  \int_0^1\sum_{j=0}^{n-1}|v_j|_{H^1(0,1)}^2 \bsplti^{(j)}(y)^2\dd y.
		\end{equation}		
		Using the definition of the norms, we obtain
		\begin{equation}\nonumber
				\int_0^1\int_0^1 \sum_{j=0}^{n-1} (v_j(x)-w_j(x))^2\bsplti^{(j)}(y)^2\dd x\dd y
						 \le 2\; \hn^2  \int_0^1\int_0^1\sum_{j=0}^{n-1} v_j'(x)^2 \bsplti^{(j)}(y)^2\dd x\dd y
		\end{equation}
		and further
		\begin{equation}\nonumber
				\|v-w\|_{L^2(\Omega)} \le \sqrt{2}\; \hn  \left\|\frac{\partial}{\partial x} v\right\|_{L^2(\Omega)}.
		\end{equation}	
		Using~\eqref{eq:2d:x}, we obtain
		\begin{equation}\label{eq:2d:2}
				\|v-w\|_{L^2(\Omega)} \le \sqrt{2}\; \hn \left\|\frac{\partial}{\partial y} u\right\|_{L^2(\Omega)}.
		\end{equation}		
		Using~\eqref{eq:2d:1} and~\eqref{eq:2d:2}, we obtain
		\begin{align}
					\|u-w\|_{L^2(\Omega)} &\le \|u-v\|_{L^2(\Omega)} + \|v-w\|_{L^2(\Omega)} \nonumber \\
						& \le \sqrt{2}\; \hn \left\|\frac{\partial}{\partial y} u\right\|_{L^2(\Omega)} +
						 \sqrt{2}\; \hn \left\|\frac{\partial}{\partial x} u\right\|_{L^2(\Omega)}\label{eq:anisotropic-estimate}\\
						& \le 2\; \hn |u|_{H^1(\Omega)},\nonumber
		\end{align}
		which finishes the proof.\qed
\end{proof}

The extension of Theorem~\ref{thrm:inverse} to two or more dimensions is rather easy.
\begin{theorem}
	For all grid sizes $h$ and each $p\in \mathbb{N}$, the inequality 
	\begin{equation}\nonumber
		|\uph|_{H^1(\Omega)} \le 2 \;\sqrt{3d} \;\hn^{-1}\;\|\uph\|_{L^2(\Omega)}
	\end{equation}
	is satisfied for all $\uph\in \widetilde{W}_{p,h}(\Omega)$.
\end{theorem}
\begin{proof}
	For sake of simplicity, we restrict ourselves to $d=2$. The generalization to more
	dimensions is completely analogous.

	We have obviously
	\begin{align*}\nonumber
		|\uph|_{H^1(\Omega)}^2 &= \left\|\frac{\partial}{\partial x} \uph\right\|_{L^2(\Omega)}^2 + \left\|\frac{\partial}{\partial y} \uph\right\|_{L^2(\Omega)}^2\\
			& = \int_0^1 |\uph(\cdot,y)|_{H^1(0,1)}^2  \dd y
						+ \int_0^1 |\uph(x,\cdot)|_{H^1(0,1)}^2  \dd x
	\end{align*}
	This can be bounded from above using Theorem~\ref{thrm:inverse} via
	\begin{align*}\nonumber
		 |\uph|_{H^1(\Omega)}^2 \leq& 12 \hn^{-2} \left( \int_0^1 \|\uph(\cdot,y)\|_{L^2(0,1)}^2  \dd y
						+ \int_0^1 \|\uph(x,\cdot)\|_{L^2(0,1)}^2  \dd x\right) \\ =& 24 \hn^{-2} \|\uph\|_{L^2(\Omega)}^2,
	\end{align*}
	which finishes the proof.\qed
\end{proof}

The extension to isogeometric spaces can be done following the approach presented in \cite{Bazilevs:2006}, Section 3.3. 
In Isogeometric Analysis, we have a geometry parameterization $\mathbf{F}:(0,1)^d \rightarrow \OmegaPhys$. An isogeometric function on $\OmegaPhys$ 
is then given as the composition of a B-spline on $(0,1)^d$ with the inverse of $\mathbf{F}$. 
The following result can be shown using a standard chain rule argument.

There exists a constant $C=C(\mathbf{F},q)$ such that 
\begin{equation}\label{eq:norm:equivalence}
	C^{-1} \left\| f \right\|_{H^{q}(\OmegaPhys)} 
	\leq \left\| f \circ \mathbf{F} \right\|_{H^{q}((0,1)^d)}
	\leq C \left\| f \right\|_{H^{q}(\OmegaPhys)}
\end{equation}
for all $f\in H^{q}(\OmegaPhys)$.

See \cite{Bazilevs:2006}, Lemma 3.5, or \cite{Beirao:2012}, Corollary 5.1, for related results. In both papers the statements are slightly more general, \cite{Bazilevs:2006} gives a more detailed dependence on the parameterization $\mathbf{F}$ whereas \cite{Beirao:2012} establishes bounds for anisotropic grids. 
Obviously, an extension to anisotropic grids can be achieved directly using the estimate \eqref{eq:anisotropic-estimate}. Note that the degree and the grid size are not necessarily equal in each parameter direction.

Using this equivalence of norms, we can transfer all results from the parameter domain $(0,1)^d$ to the physical domain $\OmegaPhys$. However, we need to point out that this equivalence is not valid for seminorms. Hence, in Theorem \ref{thrm:approx} (and follow-up Theorems \ref{thrm:approx:sob}, \ref{thrm:approx:sob:2} and \ref{eq:approx2d}) the seminorms on the right hand side of the equations need to be replaced by the full norms. Moreover, the bounds depend on the geometry parameterization via the constant $C$ in \eqref{eq:norm:equivalence}.

A similar strategy can be followed when extending the results to NURBS. We do not go into the details here but refer to \cite{Bazilevs:2006,Beirao:2012} for a more detailed study. In the case of NURBS the seminorms again have to be replaced by the full norms due to the quotient rule of differentiation. In that case the constants of the bounds additionally depend on the given denominator of the NURBS parameterization.

\section*{Acknowledgements} The authors want to thank Walter Zulehner for his suggestions, which helped to improve the
presentation of the results in this paper.

\section*{Appendix}
At this point, we want to give a basis for $\tiSph(\OmegaAB)$ to make the reader more familiar with
that space and to demonstrate that it is possible to work with it. The basis, which we introduce,
is directly related to the (scaled) cardinal B-splines $\{\bspl^{(i)}\}^{\nh-1}_{i=-p}$.

\begin{lemma}\label{lem:basis-tilde}
The set 
$\{ \bsplti^{(i)} \}^{}_{i=-\left\lceil\frac p2\right\rceil , \ldots, \nh-\left\lfloor\frac p2\right\rfloor-1}$
with 
\begin{equation}\label{eq:basis-tilde}
	\bsplti^{(i)} := \sum_{l \in \{-i-p-1,i,2\nh -i -p -1\}} \bspl^{(l)}
\end{equation}
is a basis of $\tiSph(\OmegaAB)$.
\end{lemma}
Before we prove Lemma~\ref{lem:basis-tilde} we give a more practical representation of the basis functions
by removing all contributions vanishing in $\Omega$. We obtain for odd $p$ that 
\begin{align*}
		&\bsplti^{(i)} = \bspl^{(i)}  &&   i = -(p+1)/2\\
		&\bsplti^{(i)} = \bspl^{(i)} + \bspl^{(-i-p-1)}  &&   i = -(p-1)/2,\ldots,-1 \\
		&\bsplti^{(i)} = \bspl^{(i)} && i=0,\ldots,\nh -p \\
		&\bsplti^{(i)} = \bspl^{(i)} + \bspl^{(2\nh-i-p-1)} && i=\nh-p+1,\ldots, \nh-(p+1)/2 \\
		&\bsplti^{(i)} = \bspl^{(i)} && i=\nh-(p-1)/2
\end{align*}
and for even $p$ that
\begin{align*}
		&\bsplti^{(i)} = \bspl^{(i)} + \bspl^{(-i-p-1)} &&   i = -p/2,\ldots,-1 \\
		&\bsplti^{(i)} = \bspl^{(i)} && i=0,\ldots,\nh -p-1 \\
		&\bsplti^{(i)} = \bspl^{(i)} + \bspl^{(2\nh-i-p-1)}  && i=\nh-p,\ldots, \nh-p/2-1. 
\end{align*}
Note that here we need that $0 \leq \nh-p-1$, which is equivalent to $hp<1$.
\begin{proof}{\em of Lemma~\ref{lem:basis-tilde}}
		For the sake of simplicity, we consider the case $\OmegaAB=(0,1)$ only. 
		We show first that
		every function in~\eqref{eq:basis-tilde} is in $\tiSph(0,1)$. Note that we have constructed
		$\tiSph(0,1)$ such that the restriction of any symmetric function in $\Sphper(-1,1)$
		to $(0,1)$ is a member of $\tiSph(0,1)$. Let $n=1/h$. So, consider the functions 
		$\{\bsplper^{(j)}\}_{j=-n}^{n-1}$,
		forming a basis for $\Sphper(-1,1)$. Here we consider a different indexing with $j=i-n$. Defining 
		$$s_{j}(x) := \bsplper^{(j)}(x)+\bsplper^{(j)}(-x)= \bsplper^{(j)}(x)+\bsplper^{(-j-p-1)}(x),$$
		for $j=-n,\ldots,n-1$, we obtain symmetric functions in $\Sphper(-1,1)$. Using the relation 
		$$\bsplper^{(j)}|_{(0,1)} = \sum_{k\in\mathbb{Z}} \bspl^{(j+2nk)},$$
		we obtain that the restriction of $s_j$ to $(0,1)$ fulfills 
		$$s_j|_{(0,1)} = \sum_{k\in\mathbb{Z}} \bspl^{(j+2nk)} + \sum_{k\in\mathbb{Z}} \bspl^{(-j-p-1+2nk)} 
		= \bspl^{(j)} + \sum_{k\in\mathbb{Z}} \bspl^{(-j-p-1+2nk)},$$
		which is 
\begin{align*}
		&s_j|_{(0,1)} = \bspl^{(j)} + \bspl^{(-j-p-1)}
		&&\mbox{ for } j\in \{ -n,\ldots,-1\},\\
		&s_j|_{(0,1)} = \bspl^{(j)}
		&&\mbox{ for } j\in \{ 0,\ldots,n-p-1\}, \mbox{ or}\\
		&s_j|_{(0,1)} = \bspl^{(j)} + \bspl^{(-j-p-1+2n)}
		&&\mbox{ for } j\in \{ n-p,\ldots,n-1\}. 
\end{align*}
		In all three cases $s_j$ equals $\bsplti^{(j)}$ or $2\bsplti^{(j)}$. 
		This shows that
		$\bsplti^{(i)}\in \tiSph(0,1)$.
		
		It is easy to see that the functions in~\eqref{eq:basis-tilde} are linear independent 
		for $i=-\left\lceil\frac p2\right\rceil , \ldots, n-\left\lfloor\frac p2\right\rfloor-1$. 
		So, it remains to show that every function $\uph \in \tiSph(0,1)$ can be expressed as a linear
		combination of the functions in~\eqref{eq:basis-tilde}. As we have already noticed, by construction
		the function $\uph$ can be extended to $(-1,1)$, by defining $\wph(x):=\uph(|x|)$. Note that
		$\wph \in \Sphper(-1,1)$. So, we can express it as a linear combination of basis functions
		of the basis given in~\eqref{eq:basis:varphi} via
		\begin{equation*}
				\wph = \sum_{j=-n}^{n-1} w_j \bsplper^{(j)}.
		\end{equation*}
		By construction, $\wph(x)=\wph(-x)$, so we obtain
		\begin{align*}
				\wph(x) & = \frac12 (\wph(x)+\wph(-x))
				= \frac12 \sum_{j=-n}^{n-1} w_j (\bsplper^{(j)}(x)+\bsplper^{(j)}(-x) ) \\
				& = \frac12 \sum_{j=-n}^{n-1} w_j (\bsplper^{(j)}(x)+\bsplper^{(-j-p-1)}(x) )\\
				& = \frac12 \sum_{j=-n}^{n-1}\sum_{k\in\mathbb{Z}} w_j (\bspl^{(j+2nk)}(x)+\bspl^{(-j-p-1+2nk)}(x) )\\
				& = \frac12 \sum_{j=-n}^{n-1}w_j (\bspl^{(-j-p-1)}(x)+\bspl^{(j)}(x)+\bspl^{(2n-j-p-1)}(x) ).
		\end{align*}
		Again, it can be checked easily, that for all $j,n\in\mathbb{Z}$ the term 
		$$
		\bspl^{(-j-p-1)}(x)+\bspl^{(j)}(x)+\bspl^{(2n-j-p-1)}(x)
		$$
		 is in the span of $\{ \bsplti^{(i)} \}^{}_{i=-\left\lceil\frac p2\right\rceil , \ldots, n-\left\lfloor\frac p2\right\rfloor-1}$, which concludes the proof.
\qed\end{proof}
We observe that the basis forms a partition of unity. Moreover, all basis functions are obviously
non-negative linear combinations of B-splines. Hence we call it a \emph{B-spline-like basis}.

Fig.~\ref{fig:p2} and~\ref{fig:p4} depict the B-spline basis functions that span $\tiSph(0,1)$. 
Here, the basis functions that have an influence at the boundary
are plotted with solid lines. The basis functions that have zero derivatives up to order $p-1$ 
at the boundary coincide with standard B-spline functions. They are plotted
with dashed lines.

If we compare the pictures of the B-spline basis functions in $\tiSph(0,1)$ (Fig.~\ref{fig:p2}
and~\ref{fig:p4}) with the standard B-spline basis functions for $\Sph(0,1)$ (Fig.~\ref{fig:p2a}
and~\ref{fig:p4a}) obtained from a classical open knot vector, we see that the latter ones
have more basis functions that are associated with the boundary. This can also be seen by counting
the number of degrees of freedom, cf. Table~\ref{tab:dof}.

\begin{figure}
	\begin{center}
		\includegraphics[scale=.55]{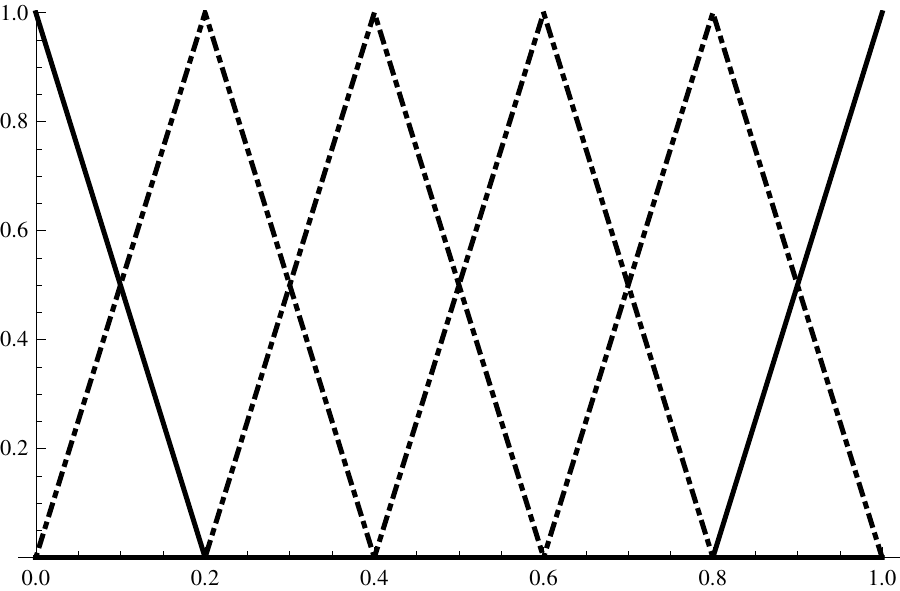}
		\includegraphics[scale=.55]{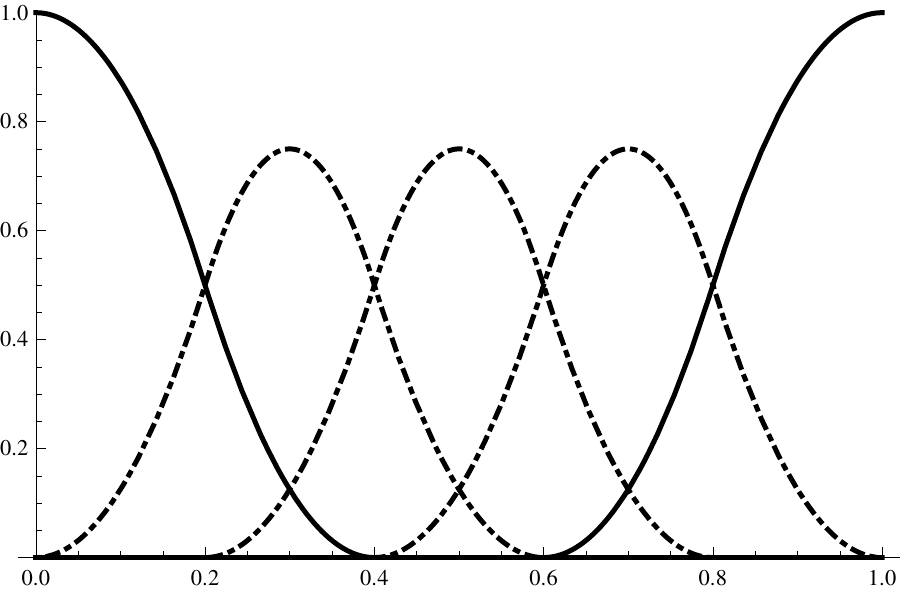}
		\caption{B-spline-like basis functions for $\tiShone(0,1)$ and $\tiShtwo(0,1)$}
		\label{fig:p2}
	\end{center}
	\begin{center}
		\includegraphics[scale=.55]{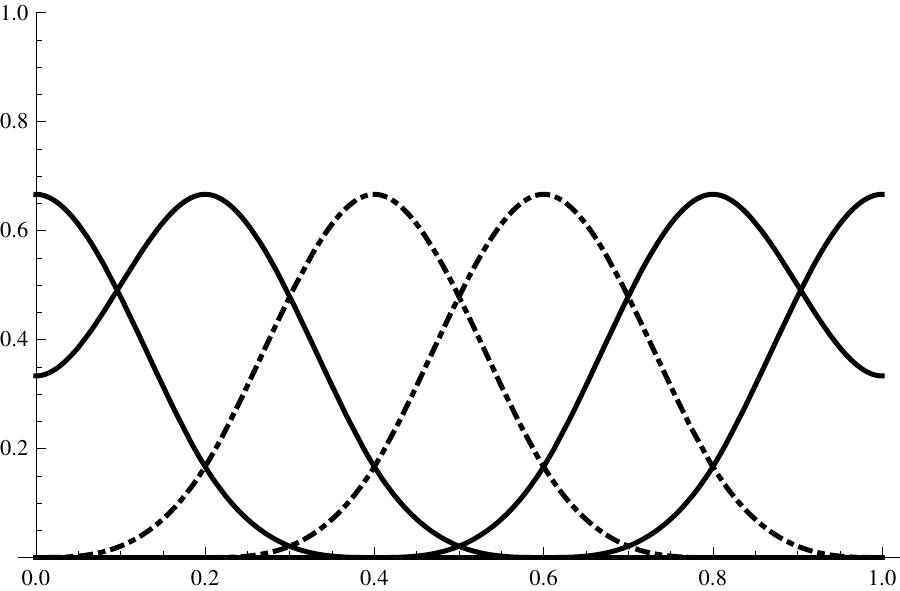}
		\includegraphics[scale=.55]{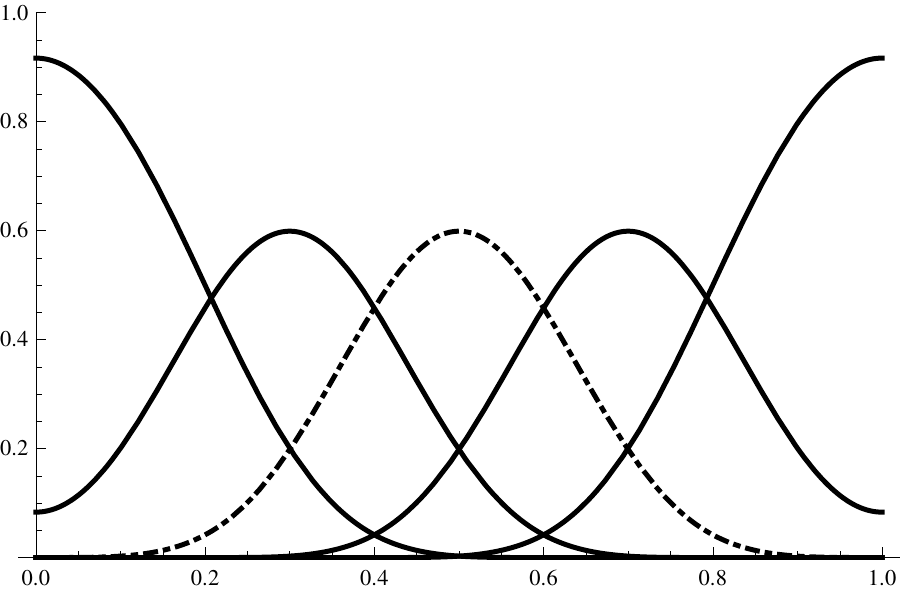}
		\caption{B-spline-like basis functions for $\tiShthree(0,1)$ and $\tiShfour(0,1)$}
		\label{fig:p4}
	\end{center}
\end{figure}
\begin{figure}
	\begin{center}
		\includegraphics[scale=.55]{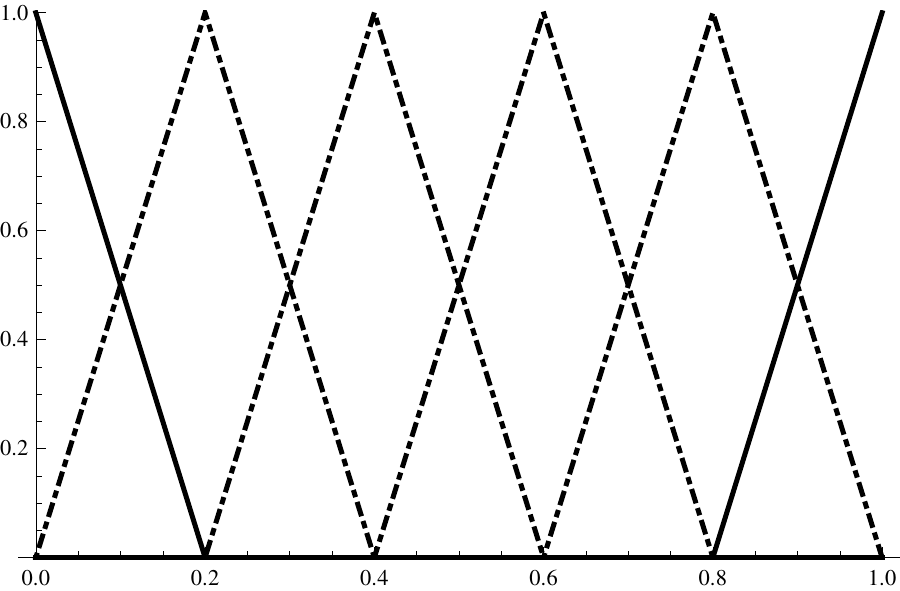}
		\includegraphics[scale=.55]{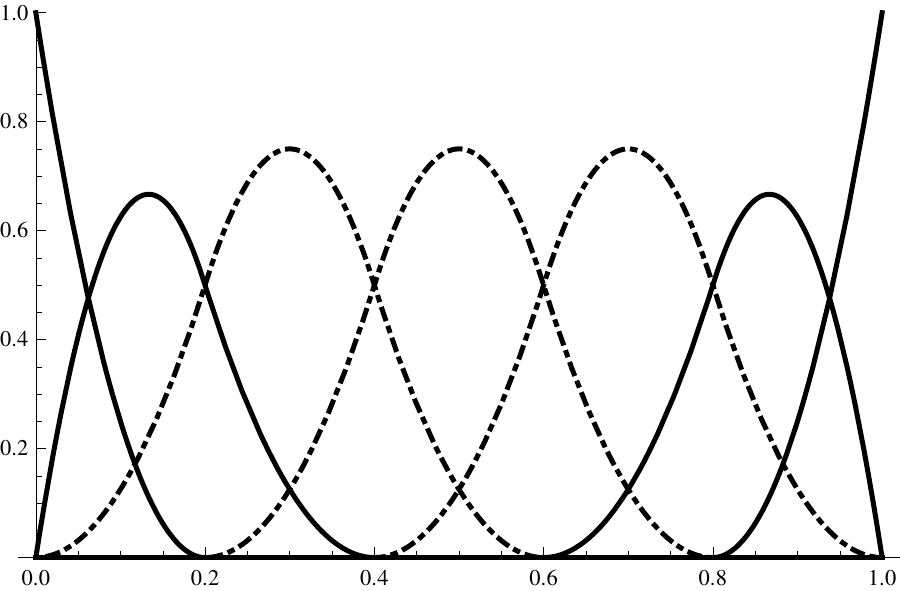}
		\caption{B-spline basis functions for $\Shone(0,1)$ and $\Shtwo(0,1)$}
		\label{fig:p2a}
	\end{center}
	\begin{center}
		\includegraphics[scale=.55]{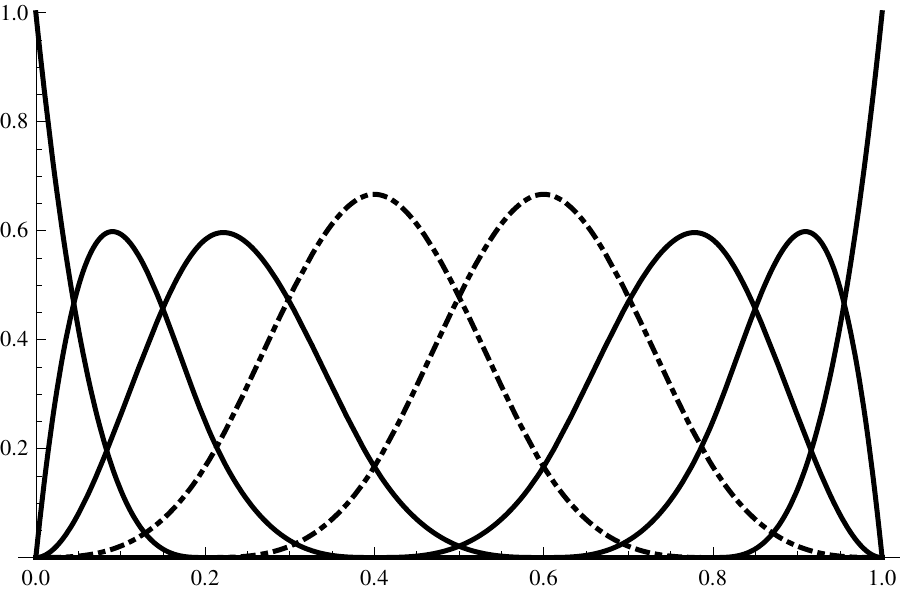}
		\includegraphics[scale=.55]{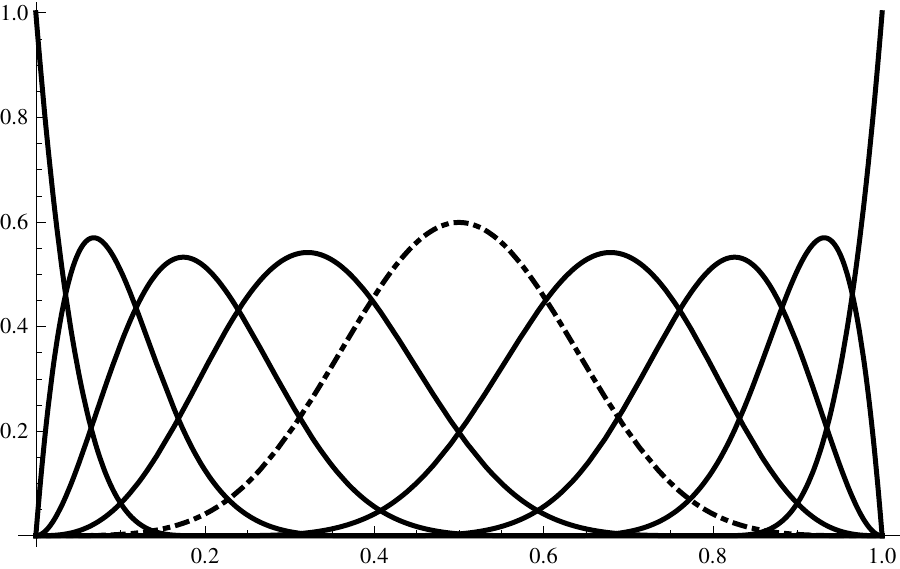}
		\caption{B-spline basis functions for $\Shthree(0,1)$ and $\Shfour(0,1)$}
		\label{fig:p4a}
	\end{center}
\end{figure}

\bibliographystyle{amsplain}
\bibliography{literature}

\end{document}